\documentclass[ejs,preprint]{imsart}
 
\usepackage{amsmath,amsthm,amssymb,amsfonts,bm,bbm,xcolor,graphicx,fancyhdr,lastpage,mathtools,comment,colordvi,graphicx} 
\usepackage[normalem]{ulem}
\usepackage[shortlabels]{enumitem}
\usepackage{cmap} \usepackage[T1]{fontenc} 
\definecolor{navy}{RGB}{20,0,105}
\usepackage[pdfauthor   = {Us}, pdftitle    = {}, pdfsubject  = {}, pdfkeywords = {}, bookmarks=true, bookmarksopen=true, pagebackref=false, hyperindex=true, colorlinks=true, allcolors=navy]{hyperref} 
\usepackage[capitalise]{cleveref}		
\renewcommand{\namecref}{\lcnamecref}

\usepackage[titletoc,title]{appendix}

\DeclareMathOperator*{\argmax}{argmax}
\DeclareMathOperator{\Var}{Var}

\DeclareMathOperator{\FDP}{FDP}
\DeclareMathOperator{\FDR}{FDR}
\DeclareMathOperator{\BFDR}{BFDR}

\DeclareMathOperator{\postFDR}{postFDR}
\DeclareMathOperator{\FNR}{FNR}
\definecolor{blendedblue}{rgb}{0.2,0.2,0.7}

\newcommand{\arxivurl}[1]{\href{www.arxiv.org/abs/#1}{#1}}

\newcommand{\ind}[1]{\II\braces{#1}}
\newcommand{\vphi}{\varphi}
\newcommand{\eps}{\varepsilon}

\newcommand{\given}{\,|\,}

\newcommand{\be}{\beta}

\newcommand{\Ga}{\Gamma}

\newcommand{\la}{\lambda}

\newcommand{\te}{\theta}

\newcommand{\Cl}{{\textnormal{C}{\ensuremath{\ell}}}}
\newcommand{\qval}{{\ensuremath{q}\textnormal{-val}}}

\theoremstyle{definition}  \newtheorem*{definition*}{Definition}  \newtheorem*{definitions*}{Definitions}  
\theoremstyle{remark}  \newtheorem*{remark}{Remark}  \newtheorem*{remarks}{Remarks}
\theoremstyle{plain} \newtheorem{theorem}{Theorem} \newtheorem*{theorem*}{Theorem} \newtheorem{lemma}[theorem]{Lemma}  \newtheorem*{lemma*}{Lemma} 

\newcommand{\dif}{\mathop{}\!\mathrm{d}} 
\newcommand{\dt}{\dif t}       \newcommand{\dx}{\dif x}

   \newcommand{\NN}{\mathbb{N}}  \newcommand{\RR}{\mathbb{R}}   \newcommand{\II}{\mathbbm{1}} 
\newcommand{\Aa}{\mathcal{A}} \newcommand{\Bb}{\mathcal{B}} \newcommand{\Cc}{\mathcal{C}}           \newcommand{\Nn}{\mathcal{N}}            

\renewcommand{\em}{\it}
\newcommand{\EM}{\ensuremath}

\newcommand{\cG}{\EM{\mathcal{G}}}

\newcommand{\cN}{\EM{\mathcal{N}}}

\DeclarePairedDelimiter{\norm}{\lVert}{\rVert}
\DeclarePairedDelimiter{\abs}{\lvert}{\rvert}
\DeclarePairedDelimiter{\braces}{ \{ }{ \} }
\DeclarePairedDelimiter{\brackets}{(}{)}
\DeclarePairedDelimiter{\sqbrackets}{[}{]}

\crefname{appsec}{Appendix}{Appendices}
\crefname{assumption}{assumption}{assumptions}
\crefname{equation}{}{}
\crefname{enumi}{}{}
\newlist{lemenum}{enumerate}{1} 
\setlist[lemenum]{label=\alph*., ref=\arabic{theorem}\alph*}
\crefalias{lemenumi}{lemma} 
\newlist{thmenum}{enumerate}{1} 
\setlist[thmenum]{label=\Alph*., ref=\arabic{theorem}\Alph*}
\crefalias{thmenumi}{theorem} 

\begin{document}
\begin{frontmatter}
\title{Empirical Bayes cumulative $\ell$-value multiple testing procedure for sparse sequences}

\runtitle{$\Cl$-value multiple testing}

\begin{aug} 
\author{\fnms{Kweku} \snm{Abraham}
\ead[label=e1]{lkwa2@cam.ac.uk}},
\author{\fnms{Isma\"el} \snm{Castillo}
\ead[label=e2]{ismael.castillo@upmc.fr}}
\and
\author{\fnms{\'Etienne} \snm{Roquain}
\ead[label=e3]{etienne.roquain@upmc.fr}}
\address{University of Cambridge \\ Statistical Laboratory \\ Wilberforce Road, Cambridge CB3 0WB, UK\\
\printead{e1}}
\address{Universit\'{e} de Paris and Sorbonne Universit\'{e},\\
	CNRS, Laboratoire de Probabilit\'{e}s, Statistique et Mod\'{e}lisation,\\
	F-75013 Paris, France\\
\printead{e2,e3}}

\runauthor{K. Abraham, I. Castillo, E. Roquain}

\affiliation{}

\end{aug}

\begin{abstract}
In the sparse sequence model, we consider a popular Bayesian multiple testing procedure and investigate for the first time its behaviour from the frequentist point of view. Given a spike-and-slab prior on the high-dimensional sparse unknown parameter, one can easily compute posterior probabilities of coming from the spike, which correspond to the well known local-fdr values \cite{ETST2001}, also called $\ell$-values. The spike-and-slab weight parameter is calibrated in an empirical Bayes fashion, using marginal maximum likelihood.
The multiple testing procedure under study, called here the \emph{cumulative $\ell$-value procedure},  ranks coordinates according to their empirical $\ell$-values and thresholds so that the cumulative ranked sum does not exceed a user-specified level $t$. 
 We validate the use of this method from the multiple testing perspective:  for alternatives of appropriately large signal strength, the false discovery rate (FDR) of the procedure is shown to converge to the target level $t$, while its false negative rate (FNR) goes to $0$. We complement this study by providing convergence rates for the method. Additionally, we prove that the $q$-value multiple testing procedure \cite{Storey2003, CR18} shares similar convergence rates in this model.

\end{abstract}



\tableofcontents

\end{frontmatter}

\section{Introduction}
\subsection{Background}

Multiple testing problems are ubiquitous and encountered in applications as diverse as genomics, imaging, and astrophysics. The seminal paper of Benjamini and Hochberg \cite{BH1995} introduced the False Discovery Rate (FDR) as a criterion for multiple testing and provided a procedure controlling it, the so-called Benjamini--Hochberg procedure. Subsequent papers  adapted this procedure in different contexts \cite{BY2001,BKY2006,BR2009,RW2009,Liu2013gaussian,Ign2016,CDH2018,Dur2017,LF2016,DDR2018,bogdanslope15,BC2015,barber2019knockoff,javanmard2019false}. We focus here on another class of multiple testing procedures, also widely used in practice, consisting of empirical Bayesian procedures. 
These have been made popular in particular through the two-group model \cite{ETST2001} and a series of papers by Efron \cite{Efron2004, Efron2007, Efron2008}; see also  \cite{SC2007,AS2015,SS2018} for several extensions. More specifically, the local FDR (called $\ell$-value here) can be seen as a Bayesian quantity corresponding to the probability of being under the null distribution conditionally on the value of the test statistic. This probability is typically estimated by plugging in estimators of model aspects, which follows the general philosophy of empirical Bayes methods. Using $\ell$-values instead of $p$-values is often considered to be more powerful \cite{SC2007}, which explains the popularity of these significance measures in practical applications, including genomic data and biostatistics 
\cite{muelleretal04, encode2007,zablocki2014,JY2016,amar2017extracting, Steph2016, GS2018} 
 but also other applied fields, such as neuro-imaging as in e.g.\ \cite{lee2016}. 
In addition, the detection ability of $\ell$-values can be increased further by  adding structure on the null configurations via a latent model, such as a hidden Markov model \cite{SC2009,AGC20} or a stochastic block model \cite{RRV2019}, or via covariates \cite{SC2019}. 

Despite their popular practical use,  Bayesian multiple testing methods remain much less understood from the theoretical point of view than $p$-value based approaches.  Decision-theoretic arguments inspire most practical algorithms based on the Bayesian distribution (see among others $\ell$-, $\Cl$- and $q$-value procedures defined below). Such arguments are theoretically justified under the assumption that the data has been generated from a model which includes specific random modelling of the latent parameters, and this random modelling can be seen as a Bayesian prior. Yet, in practice, especially in sparsity problems, specification of prior aspects such as the number of effective parameters and the distribution of alternative means is delicate. In the frequentist-Bayes literature,  an alternative is to look for prior distributions that can be proved to have optimal or near-optimal behaviour from the frequentist point of view (see Section \ref{sec:twom} below for general references). This leads to the question of studying  Bayesian multiple testing procedures in the frequentist sense. 
From the perspective of multiple testing theory, the goal is to design procedures that are robust with respect to the latent modelling, which is in line with the classical strong error rate control \cite{Dic2014}.

While most of the literature on multiple testing for Bayesian methods has focused on latent variable modelling with a random `signal' parameter, we thus focus here on the case of any deterministic signal. There are very few works so far in this setting --- we present a brief literature review in Section \ref{sec:twom} --- and the present work can be seen as a continuation of \cite{CR18}. In that work, a family of spike-and-slab prior distributions was considered and frequentist properties of two multiple testing procedures were investigated in the sparse sequence model: the $\ell$-value procedure, where testing is based on the posterior probability that a given null hypothesis is true, and the $q$-value procedure, based on the Bayesian probability of the null given the hypothetical event that the data exceeds the value actually observed. A different  procedure very popular in practice is one based on cumulative ranked $\ell$-values, called the $\Cl$-value procedure below. This procedure was conjectured to have desirable frequentist properties in \cite{CR18}. The aim of the present paper is to confirm this conjecture: the \Cl-value procedure is studied here for the first time from the frequentist perspective in the setting of sparse deterministic hypotheses. We now proceed to introducing in more detail the model, the inferential goals, and the multiple testing procedures to be considered.

\subsection{Model, FDR and FNR} \label{sec:mod}
Notation introduced throughout the paper is collected in \cref{sec:notation} for the reader's convenience. 

{\em Model.} 
Consider the Gaussian sequence model, for $\te=(\te_1,\ldots,\te_n)\in\RR^n$,
\begin{equation}
\label{eqn:def:GaussianSequenceModel}
X_i=\theta_{i} + \eps_i, \quad 1\leq i \leq n,
\end{equation}
where the noise variables $(\eps_i)_{i\leq n}$ are assumed to be iid standard Gaussians $\cN(0,1)$, whose density we denote $\phi$. We assume that there exists a true (unknown) vector $\te_0\in\RR^n$ that is {\em sparse}: specifically, if 
\[\norm{\theta}_{\ell_0}:=\# \braces{1\leq i\leq n : \theta_i\not = 0}\]
denotes the number of non-zero coordinates of $\theta$, we assume 
 that $\theta_0\in \ell_0(s_n)$ for a sequence $s_n\to \infty$ satisfying $s_n/n\to 0$ as $n\to \infty$, where for $s\geq0$,
\begin{equation}\label{eqn:def:SparseSequenceDefinition}
\ell_0(s)=\braces{\theta\in\RR^n : \norm{\theta}_{\ell_0} \leq s}.
\end{equation}
The distribution of the data under the true $\theta_0$ is given by
 \[ P_{\te_0} = \bigotimes_{i=1}^n\,\cN(\te_{0,i},1),\]
 where $\te_0$ satisfies the sparsity constraint \eqref{eqn:def:SparseSequenceDefinition} but is otherwise arbitrary and non-random. 
  To make inference on $\theta$, we follow a Bayesian approach and endow $\te$ with a prior distribution $\Pi$. Using Bayes' formula one can then form the posterior distribution $\Pi[\cdot\given X]$, which is the conditional distribution of $\te$ given $X$ in the Bayesian framework. The choice of $\Pi$ (and the corresponding posterior $\Pi[\cdot\given X]$) will be specified in more detail in Section \ref{sec:prior} below. To assess the validity of inference using $\Pi[\cdot\given X]$, we study the behaviour of the latter --- or of aspects of it used to build a testing procedure --- in probability under the true frequentist distribution $P_{\te_0}$.\\

{\em Multiple testing inferential problem, FDR and FNR.} 
We consider the multiple testing problem of determining for which $i$ we have signal, that is, $\theta_{0,i}\not=0$. More formally, we analyse a procedure $\vphi(X)=(\vphi_i(X))_{1\le i\le n}$, taking values in $\{0,1\}^n$, that for each coordinate $i$ guesses whether or not signal is present. To evaluate the quality of such a procedure $\vphi$, one needs to consider certain risk or loss functions. Here we focus on the most popular such risks, defined as follows: the FDR (false discovery rate) is the average proportion of errors among the positives, while the FNR (false negative rate) is the average proportion of errors among the true non-zero signals. 

More precisely, first define  the false discovery proportion (FDP) at $\theta_0$ by
\begin{equation} \label{eqn:def:FDP}
\FDP(\vphi;\theta_0):= \frac{\sum_{i=1}^{n} \II\braces{\theta_{0,i}=0, \vphi_i=1}}{1\vee \brackets[\big]{\sum_{i=1}^{n} \vphi_i}}.
\end{equation}
Then the FDR at $\theta_0$ is given by
\begin{equation}\label{eqn:def:FDR}
\FDR(\vphi;\theta_0):= E_{\theta_0} [\FDP(\vphi;\theta_0)].
\end{equation}
Similarly, the false negative rate (FNR) at $\theta_0$ is defined as
\begin{equation}\label{eqn:def:FNR}
\FNR(\vphi;\theta_0):= E_{\theta_0}\sqbrackets[\bigg]{\frac{\sum_{i=1}^{n} \II\braces{\theta_{0,i}\neq 0, \vphi_i=0}}{1\vee\brackets[\big]{\sum_{i\leq n} \II\braces{\theta_{0,i}\neq 0}}}}.
\end{equation}
To use classical testing terminology, the FDR can be interpreted as a type I error rate, while the FNR corresponds to a type II error rate. The former is ubiquitous and the latter, with the current (non-random) choice of denominator, has been widely used in recent contributions, see, e.g., \cite{AC2017,RRJW2020}.

The aim of multiple testing in this setting is to find procedures that keep both type of errors under control. Inevitably, in the sparse setting in model \eqref{eqn:def:GaussianSequenceModel}, to achieve this will  require some signal strength assumption (see  \eqref{eqn:def:StrongSignalClass} below and the discussion in \cref{sec:optimality}).

\subsection{Empirical Bayes multiple testing procedures} \label{sec:proc}
 
 \subsubsection{Spike-and-slab prior distributions and empirical Bayes} \label{sec:prior}

{\em A family of prior distributions.} 
For $w\in(0,1)$, let $\Pi_w=\Pi_{w,\gamma}$ denote a spike-and-slab prior for $\theta$, where, for $\Ga$ a distribution with density $\gamma$,
\begin{equation}
\label{eqn:def:SpikeAndSlabPrior}
\Pi_w = ((1-w)\delta_0+w\Ga)^{\otimes n}. 
\end{equation} That is, under $\Pi_w$, the coordinates of $\theta$ are independent, and are either exactly equal to $0$, with probability $(1-w)$, or are drawn from the `slab' density $\gamma$. When the Bayesian model holds, the data $X$ follows a mixture distribution, with each coordinate $X_i$ independently having density $(1-w)\phi + w g$, where $g$ denotes the convolution $\phi \star \gamma$. This shares similarities with the well-known two-group model in the multiple testing literature \cite{ETST2001}: the only difference is that here the  alternative (i.e.\ the slab), is fixed a priori, {rather than estimated from the data}. 

In this work, we consider in particular a `quasi-Cauchy' alternative as in \cite{js05}, where $\gamma$ is defined in such a way that the convolution $g=\phi\star\gamma$ equals
\begin{equation} \label{eqn:def:gInQuasiCauchy} g(x)=(2\pi)^{-1/2}x^{-2}(1-e^{-x^2/2}),\:\:\: x\in\RR.
\end{equation}
Such a $\gamma$ indeed exists, and is given explicitly in \cite{js05}, eq.\ (4), but its explicit expression will not be of use to us here.

The references \cite{JS04,CR18} consider more generally a family of heavy-tailed distributions governed by a parameter $\kappa\in [1,2]$, for which the quasi-Cauchy alternative corresponds to $\kappa=2$, and we note that most of the calculations in the current paper work unchanged in the Laplace case $\kappa=1$. Some, however, require minor adjustment, and in particular, one should expect a slightly different rate of convergence of the FDR to $t$ in \cref{thm:FDRConvergenceWithRate,thm:QvalueControlsFDR} below.

The posterior distribution $\Pi_w(\cdot \mid X)$ can be explicitly derived as 
\begin{equation*}
\theta \:|\: X\,\sim\, \bigotimes_{i=1}^n\  \left(\ell_{i,w}(X)\, \delta_0 + (1-\ell_{i,w}(X))\, \cG_{X_i}\right),
\end{equation*}
where $\cG_{x}$ is the distribution  with density $\gamma_{x}(u) := \phi(x-u) \gamma(u)/g(x)$ and
\begin{equation}
\label{eqn:def:Lvals}
\begin{split}
\ell_{i,w}(X)&=\Pi_w(\theta_i=0 \mid X)= \ell(X_i;w),\quad 1\leq i\leq n,
\\ \ell(x;w) &= \frac{(1-w)\phi(x)}{(1-w)\phi(x) + w g(x)}\in (0,1), \:\:\:x\in\RR.
\end{split}
\end{equation} 

The quantities $\ell_{i,w}(X)$, $1\leq i \leq n$, are called the \emph{$\ell$-values}. Note that $w\to \ell_{i,w}(X)$ is decreasing. For short, we sometimes write $\ell_{i,w}$ for $\ell_{i,w}(X)$. In words, each $\ell_{i,w}(X)$ corresponds to the posterior probability that the measurement $X_i$  comes from the null, this probability being computed in the Bayesian model with the spike-and-slab prior \eqref{eqn:def:SpikeAndSlabPrior}. Let us underline that, in the usual multiple testing terminology of the two-group model, the posterior distribution $\ell_{i,w}(X)$ corresponds to the $i$th local fdr of the data, when the alternative density is $g$, the null density is $\phi$, and the proportion of true nulls is $1-w$, see, e.g.,   \cite{Efron2008}.

In the empirical Bayes framework, one first estimates $w$ empirically from the data using, for example, the maximum (marginal) likelihood estimator, defined as the maximiser (which exists almost surely, in view of \cref{lem:monotonicity}) 
\begin{equation} \label{eqn:def:wHat}
\hat{w}=\argmax_{w\in [1/n,1]} L(w),\end{equation}
where $L(w)$ denotes the marginal log-likelihood function for $w$, which can be expressed as
\begin{equation}
	\label{eqn:def:logLikelihood}
	L(w)= \sum_{i=1}^n \log \brackets[\big]{(1-w)\phi(X_i)+ wg(X_i)}. 
\end{equation}
The resulting empirical Bayes (EB) posterior is simply $\Pi_{\hat w}[\cdot\given X]$. We highlight that the dependence on $w$ is a significant qualitative difference between $\ell$-values and their main alternative of $p$-values. Quantitatively, with an estimated $\hat{w}$, $\ell$-values under the null are expected to be close to $1$ (see, e.g., \cref{lem:ExistsLambda+-}), which is not the case for $p$-values.

Finding a maximiser $\hat{w}$ and simulating from this distribution, or calculating aspects such as the posterior mean or median, can be done in a fast and efficient way and has been implemented in the {\tt EBayesThresh} R package. From the theoretical perspective, a lot of progress has been made in the last few years in understanding the behaviour of the empirical Bayes posterior, in connection with the study of Bayesian procedures in sparsity settings, and we briefly review such results in Section \ref{sec:twom} below.

\subsubsection{Bayesian multiple testing procedures}\label{sec:Bayesian-multiple-testing-procedures}

The $\Cl$-value procedure with level $t\in (0,1)$, which is the main object of study herein, rejects the null hypothesis $H_{0,i} :$ ``$\theta_{0,i}=0$'' if $\ell_{i,\hat{w}}(X)<\hat{\lambda}(\hat{w},t)$, where, for $\hat{w}$ as in \cref{eqn:def:wHat},
\[ \hat{\lambda}(\hat{w},t)= \sup \braces[\Big]{\lambda \in [0,1] : \frac{\sum_{i=1}^n \ell_{i,\hat{w}}(X) \II\braces{\ell_{i,\hat{w}}(X)<\lambda}}{1\vee \sum_{i=1}^n \II\braces{\ell_{i,\hat{w}}(X)<\lambda}}\leq t}.\]
Let us recall the notions of the Bayesian FDR and posterior FDR under a prior distribution (as in e.g.\  \cite{sarkaretal08}) to place this definition in context and define the other two Bayesian multiple testing procedures of interest for this work.\\
 
{\em BFDR and postFDR.}
The Bayesian FDR is the FDR when instead of having a fixed $\theta_0$, the parameter $\theta$ is truly generated from the prior $\Pi_w$:
\begin{equation}\label{eqn:def:BFDR}
\BFDR_w(\vphi):=E_{\theta\sim \Pi_w} \FDR(\vphi;\theta),
\end{equation} 
and the posterior FDR is the BFDR conditional on $X$, or equivalently the expectation of the FDP when the parameter $\theta$ is drawn from the posterior $\Pi_w(\cdot \mid X)$: 
\begin{equation}\label{eqn:def:postFDR} \postFDR_w(\vphi) := E_{\theta\sim \Pi_w(\cdot \mid X)} [\FDP(\vphi;\theta)] 
	=\frac{\sum_{i=1}^n \ell_{i,w}(X) \vphi_i}{1\vee(\sum_{i=1}^n \vphi_i)}. \end{equation}
Note that $\postFDR_w (\vphi)$ decreases as $w$ increases (for a fixed procedure $\vphi$), as a result of the monotonicity of the $\ell$-values (see \cref{lem:monotonicity}).\\

{\em $\ell$-value procedure.} 
Let us consider a family of multiple testing procedures $\vphi=\vphi_{\lambda,w}$ based on $\ell$-value thresholding as follows. For any given level $\lambda\in [0,1]$, set 
\begin{equation} \label{eqn:def:PhiLambda} 
	\vphi_{\lambda,w}(X)=(\II\braces{\ell_{i,w}(X)<\lambda})_{1\leq i\leq n}.
\end{equation}
The $\ell$-value procedure at level $t$ is then defined by $\vphi_{t,\hat{w}}(X)$, for $\hat{w}$ as in \cref{eqn:def:wHat}.\\

{\em \Cl-value procedure (reformulation).} 
 Given the collection of procedures \eqref{eqn:def:PhiLambda} for different thresholds $\la$, another way to choose $\lambda$ is to ensure the posterior FDR \eqref{eqn:def:postFDR} is controlled at a level as close as possible to the target level $t$. 
 This yields the \Cl-value procedure defined at the start of this section: 
 with $\hat{w}$ as in \cref{eqn:def:wHat},
\begin{equation}\begin{split} 
		\vphi^{\Cl}&=\vphi_{\hat{\lambda},\hat{w}},\\ \label{eqn:def:lambdahat} \hat{\lambda}
		&=\hat{\lambda}(\hat{w},t)=\sup\braces{\lambda\in [0,1]: \postFDR_{\hat{w}} (\vphi_{\lambda,\hat{w}}) \leq t}.\end{split}
		\end{equation}
This is also a reformulation of the procedure considered in, e.g., \cite{muelleretal04,SC2007}. 
The original expression of $\vphi^{\Cl}$  {in these references (using cumulative sums rather than the level $\lambda$)} can be derived from the observation that we necessarily threshold at one of the observed $\ell$-values (i.e.\ at some $\ell_{i,\hat{w}}(X)$) since the posterior FDR only changes when we cross such a value. The threshold is $\hat{\lambda}=\ell_{(\hat{K}+1),\hat{w}}
$, with $\ell_{(i),\hat{w}}
$ denoting the $i$th order statistic of $\braces{\ell_{i,\hat{w}}(X): 1\leq i \leq n}$, and we therefore reject the null hypotheses for indices corresponding to the $\hat{K}$ smallest observed $\ell$-values,\footnote{In principle we define the order statistics so that repeats are allowed, defining them by the traits $\braces{\ell_{i,\hat{w}(X)
		}, i\leq n}=\braces{\ell_{(j),\hat{w}
		}, j\leq n}$ as a multiset ($\forall x\in\RR$, $\#\braces{i : \ell_{i,\hat{w}}
		=x}=\#\braces{i : \ell_{(i),\hat{w}}
		=x}$) and $\ell_{(1),\hat{w}}
	\leq \ell_{(2),\hat{w}}
	\leq \dots \leq \ell_{(n),\hat{w}}
	$. 
 When $\ell_{(\hat{K}),\hat{w}}
 =\ell_{(\hat{K}-1),\hat{w}}
 $ in fact $\vphi^\Cl$ as defined in \cref{eqn:def:lambdahat} rejects fewer than $\hat{K}$ hypotheses. However, with probability 1, the observed $\ell$ values are all distinct, due to the Gaussianity of $X_i$ and the strict increasingness of the map $x\mapsto \ell(x;w)$, see \cref{lem:monotonicity}.} where $\hat{K}$ is defined by
\begin{equation}
	\label{eqn:def:HatK}
	\frac{1}{\hat{K}} \sum_{i=1}^{\hat{K}} \ell_{(i),\hat{w}}
	 \leq t < \frac{1}{\hat{K}+1} \sum_{i=1}^{\hat{K}+1} \ell_{(i),\hat{w}}.
\end{equation}
(By convention the left inequality automatically holds in the case $\hat{K}=0$. If the right inequality is not satisfied for any $\hat{K}< n$, we set $\hat{K}=n$ and $\hat{\lambda}=1$.) Note that $\hat{K}$ is well defined and unique, by monotonicity of the average of nondecreasing numbers. This monotonicity also makes clear the following dichotomy, which will prove useful in the sequel: for all $t\in (0,1)$ and $\lambda\in [0,1]$,
\begin{equation}\label{eqn:HatLambdaPostFDRRelation} \postFDR_{\hat{w}}(\vphi_{\lambda,\hat{w}})\leq t \iff \lambda \leq \hat{\lambda}.\end{equation}
This indicates that the supremum in \eqref{eqn:def:lambdahat} is a maximum. Also observe that $\postFDR_{\hat{w}} (\vphi_{t,\hat{w}}) \leq t$, so that $\hat{\lambda}\geq t$ and the $\ell$-value procedure is always more conservative than the \Cl-value procedure.\\

{\em $q$-value procedure.} 
Another way to calibrate a procedure $\vphi_i=\ind{|X_i|\geq x}$ in order to control the (B)FDR is to further simplify the expectation of a ratio defining the BFDR and instead consider the ratio of expectations, defining for $x\in\RR$ and $w\in[0,1]$
\begin{equation}\label{eqn:def:Qfunction} \begin{split}
q(x;w)&=\frac{E_{\theta\sim \Pi_w} \sum_{i=1}^n \ind{\theta_i=0} \ind{\abs{X_i}\geq \abs{x}}  }{E_{\theta\sim \Pi_w} \sum_{i=1}^n  \ind{\abs{X_i}\geq \abs{x}} } \\ &=\frac{(1-w)\overline{\Phi}(\abs{x})}{(1-w)\overline{\Phi}(\abs{x})+w\overline{G}(\abs{x})},
\end{split}
\end{equation}
 where $\overline{\Phi}$ and $\overline{G}$ denote the upper tail functions of the densities $\phi$ and $g$ respectively. 
 The $q$-values are then given by 
 \begin{equation}
\label{eqn:def:qvals}
q_{i,w}(X)= q(X_i;w)=\frac{(1-w)\overline{\Phi}(|X_i|)}{(1-w)\overline{\Phi}(|X_i|)+w\overline{G}(|X_i|)}, \quad 1\leq i\leq n,
\end{equation} 
and the $q$-value procedure is defined by thresholding the $q$-values at the target level $t>0$: 
\begin{equation} \label{eqn:def:Phiqvalue} 
	\vphi^{\qval}(X)=(\II\braces{q_{i,\hat{w}}(X)< t})_{1\leq i\leq n}.
\end{equation}
Thanks to monotonicity of both the $q$ and $\ell$ values (see Lemma~\ref{lem:monotonicity}) $\vphi^{\qval}$ lies in the class \eqref{eqn:def:PhiLambda}, so that $\vphi^{\qval}=\vphi_{\la_q,\hat{w}}$ for some $\la_q=\la_q(\hat{w},t)$. As with the $\ell$-values, we sometimes write $q_{i,w}$ for $q_{i,w}(X)$. \\ 

{\em Rationale behind these procedures for FDR control.} Let us now give some intuition behind the introduction of such procedures. Consider $\vphi_{t,w},$ $\vphi_{\hat{\la}(w,t),w},$ and $\vphi_{\la_q(w,t),w}$; that is, the $\ell$-, \Cl- and $q$-value procedures respectively, but with a {\em fixed} value of $w$. All three control the Bayesian FDR (BFDR) at level $t$ under the prior $\Pi_w$: for the first and third procedures, see Proposition 1 in \cite{CR18}; for the \Cl-value procedure with fixed $w$,  since $\postFDR_{w} (\vphi_{\hat{\lambda}(w,t),w}) \leq t$, we directly have $\BFDR_w(\vphi_{\hat{\lambda}(w,t),w}) \leq t$ by taking expectations. Moreover, by concentration arguments and appealing again to \cite[Proposition 1]{CR18}, we have $\BFDR_w(\vphi_{\hat{\lambda}(w,t),w}) \approx t$ and $\BFDR_w(\vphi_{\lambda_q(w,t),w})\approx t$. Hence, from the decision-theoretical perspective, if the prior $\Pi_w$ is ``correct'', these procedures are {\em bona fide} for the purpose of controlling the BFDR. Note that this says nothing when the procedures are constructed using a random $w$ which is typically what is done in practice (as in \cref{eqn:def:wHat}). 
In addition, to derive frequentist properties, the procedure has to be evaluated under a fixed truth $\te_0$, which makes it even further from the previous decision-theoretic argument. Yet, one can expect that for $n$ large, $\hat{w}$ (and consequently the plug-in posterior $\Pi_{\hat{w}}[\cdot\given X]$) concentrates in an appropriate way, giving the hope, validated by \cref{thm:FDRConvergenceNoRate} below for the $\Cl$-value procedure with strong signals, that the frequentist FDR at $\te_0$ can still be controlled.  \\

%

\section{Main results} \label{sec:mainres}

\subsection{Consistency}

Let us define a `strong signal class' of parameters with exactly $s_n$ non-zero entries, each of which is ``large''. For $\theta_0\in\ell_0(s_n)$, denote by $S_0$ the support of $\theta_0$,
\begin{equation}\label{eqn:def:S0} S_0=\braces{i : \theta_{0,i}\neq 0}.\end{equation}
For a sequence $v_n\to \infty$ we define the strong signal class
\begin{equation}  \label{eqn:def:StrongSignalClass}
	 \ell_0(s_n;v_n) =
\braces[\big]{ \theta_0\in \ell_0(s_n) \: : \: \abs{\theta_{0,i}} \geq \sqrt{2\log(n/s_n)}+v_n  \text{ for } i\in S_{0},~ \abs{S_{0}}=s_n }.
\end{equation}
\begin{theorem}\label{thm:FDRConvergenceNoRate}

Fix $t\in(0,1)$. Consider any sequence $s_n\to \infty$ such that $s_n/n\to 0$, and any sequence $v_n\to \infty$.
	 Then, as $n\to \infty$, \begin{align}\label{eqn:FDRTendsToT} 
	 \sup_{\theta_0\in \ell_0(s_n,v_n)}|\FDR(\vphi^{\Cl};\theta_0) -t|&\to 0,\\
	 	\sup_{\theta_0\in \ell_0(s_n,v_n)} \FNR(\vphi^{\Cl};\theta_0) &\to 0. \label{eqn:FNRTendsTo0}
	 	\end{align} 

\end{theorem}
Let us emphasise that the conclusion of \cref{thm:FDRConvergenceNoRate} does not mention the prior, holding for any deterministic $\theta_0$ in the strong signal class, not only for non-zero entries of $\theta_0$ drawn from the quasi-Cauchy distribution \cref{eqn:def:gInQuasiCauchy}. Moreover, this frequentist consistency result holds \emph{uniformly} across the strong signal set $\ell_0(s_n;v_n)$. The assumption $v_n\to \infty$ cannot be relaxed, as we discuss in \cref{sec:optimality}.
%
%
\subsection{Convergence rate}\label{sec:QValueRate}

The following result  strengthens the conclusion of \cref{thm:FDRConvergenceNoRate}, showing that the FDR converges to $t$ from above and obtaining a precise rate of convergence, at the cost of requiring mild extra conditions on $s_n$ and $v_n$.

\begin{theorem}\label{thm:FDRConvergenceWithRate}
	In the setting of \cref{thm:FDRConvergenceNoRate}, assume also that \[s_n\geq (\log n)^3,\qquad v_n\geq 3(\log \log (n/s_n))^{1/2}.\]	
Then there exist constants $c,C,C'>0$ depending on $t$ such that uniformly over $\theta_0\in \ell_0(s_n;v_n)$, for all $n$ large enough we have 
	\begin{align}\label{eqn:FDRUpperAndLowerBound}c\frac{\log\log(n/s_n)}{\log(n/s_n)} \leq  \FDR(\vphi^{\Cl};\theta_0)-t &\leq C\frac{\log\log(n/s_n)}{\log(n/s_n)},\\
		\label{eqn:FNRControl} 	\FNR(\vphi^{\Cl};\theta_0)&\leq C' \brackets{\log \tfrac{n}{s_n}}^{-1}.
	\end{align}
	\end{theorem}
	
	\begin{remarks}
		\begin{enumerate}[i.]
		\item This result shows that the convergence rate of $\FDR(\vphi^{\Cl};\theta_0)$ to $t$ is logarithmic in $n/s_n$ and uniform over $\ell_0(s_n;v_n)$. In particular, note that increasing the signal strength $v_n$ does not accelerate the convergence rate of the FDR. By contrast, we provide no lower bound for the FNR in \eqref{eqn:FNRControl}, so the convergence rate of the FNR to zero can (and will) be much faster for larger $v_n$.
		\item In fact we prove the stronger false discovery \emph{proportion} result (implying \cref{eqn:FDRUpperAndLowerBound}) that for some $c,C>0$, writing $\eps_n=\log\log(n/s_n)/(\log (n/s_n))$ we have \[c\eps_n\leq \FDP(\vphi^{\Cl};\theta_0)-t\leq C \eps_n, \text{ with probability at least } 1-o(\eps_n),\]  and correspondingly for the false negative proportion. 
			\item The bound $s_n\geq (\log n)^3$ can be relaxed to 
			$s_n\geq b(\log n)^2/\log\log n$ for some large enough constant $b=b(t)$: see \cref{lem:c1alpha}. 
		\end{enumerate}
	\end{remarks}


Let us now turn to study the $q$-value procedure. The next result shows that its behaviour matches that of the \Cl-value procedure.

\begin{theorem}	\label{thm:QvalueControlsFDR}
	\Cref{thm:FDRConvergenceNoRate,thm:FDRConvergenceWithRate} continue to hold when the $\Cl$-value procedure $\vphi^{\Cl}$ is replaced by the $q$-value procedure $\vphi^{\qval}$.
\end{theorem}

The FDR/FNR of the $q$-value procedure was studied in \cite{CR18}, but without convergence rates, and  an improvement with respect to the  \Cl-value procedure was conjectured following simulations. Theorem~\ref{thm:QvalueControlsFDR} addresses this issue by showing the convergence rate is in fact exactly the same.
	Further comparisons between our results and those of \cite{CR18} will be provided in Section~\ref{sec:twom} .%

\subsection{Sketch proof of Theorems~\ref{thm:FDRConvergenceNoRate} and \ref{thm:FDRConvergenceWithRate}}\label{sec:SketchProof}

The proof relies on the concentration of $\hat{w}$ and $\hat{\lambda}$.
One shows (\cref{lem:ExistsW+-,lem:HatWConcentrates}) 
that $\hat{w}$ concentrates near a (deterministic) value $w^*$, of order slightly larger than $s_n/n$, that roughly maximizes  the expectation of the log-likelihood  \eqref{eqn:def:logLikelihood}.
Recalling that $S_0=\braces{i : \theta_{0,i}\neq 0}$ denotes the support of $\theta_0$, the signal strength assumption 
ensures that for $i\in S_0$, with high probability $\ell_{i,w^*}(X)\approx 0$. Hence, using that $\hat{\lambda}\geq t>0$,  we obtain 
\[
\sum_{i\in S_0} \vphi^\Cl_i \approx s_n; \mbox{ and more precisely } \sum_{i\in S_0} (1-\vphi^\Cl_i) = o( s_n),
\]  
see Lemma~\ref{lem:lvalsSmallForTrueSignals}. This implies that the FNR of $\vphi^{\Cl}$ tends to 0. For the FDR result, let $V_{\lambda,w}$, $\lambda,w\in [0,1]$ denote the number of false discoveries made by $\vphi_{\lambda,w}$; that is, \begin{equation}\label{eqn:def:VLambdaW}
	V_{\lambda,w}=\sum_{i\not \in S_0} \II\braces{\ell_{i,w}(X)<\lambda}.
\end{equation} 
One shows (\cref{lem:HatLambdaConcentrates,lem:ExistsLambda+-}) that with high probability, $\hat{\lambda}$ is close to the solution $\lambda^*$ to 
\[ E[V_{\lambda^*,w^*}] (E_{\theta_0=0}[\ell_{1,w^*}(X) \mid \ell_{1,w^*}(X)<\lambda^*]-t) = t s_n.\]
This is because $V_{\lambda^*,w^*}$ and $\sum_{i\not \in S_0} \ell_{i,w^*}(X)\II\braces{\ell_{i,w^*}(X) < \lambda^*}$ concentrate around their means (\cref{lem:VLambdaWConcentrates} and the proof of \cref{lem:HatLambdaConcentrates}) ensuring that with high probability (recall $\ell_{i,w^*}(X)\approx 0$ for $i\in S_0$)
\begin{align*}
 \postFDR_{w^*} (\vphi_{\lambda^*,w^*}) &\approx \frac{\sum_{i\not \in S_0} \ell_{i,w^*}(X)\II\braces{\ell_{i,w^*}(X)<\lambda^*}}{s_n+\sum_{i\not\in S_0}\II\braces{\ell_{i,w^*}(X)<\lambda^*}} \\
 &\approx \frac{  \sum_{i\notin S_0} E[ \ell_{i,w^*}(X) \II\braces{ \ell_{i,w^*}(X)<\lambda^* }]}{E[V_{\lambda^*,w^*}]+s_n} \\
 &\quad =
 \frac{ E[V_{\lambda^*,w^*}] E_{\theta_0=0}[\ell_{1,w^*}(X)\mid \ell_{1,w^*}(X)<\lambda^*]}{E[V_{\lambda^*,w^*}]+s_n}= t.
 \end{align*}
Then, again using concentration of $V_{\lambda^*,w^*}$,
\begin{align*}
	\FDR(\vphi^{\Cl};\theta_0)\approx& \frac{E[V_{\lambda^*,w^*}]}{s_n+E[V_{\lambda^*,w^*}]} \\ &= \frac{ts_n/(E_{\theta_0=0}[\ell_{1,w^*}(X)\mid \ell_{1,w^*}(X)<\lambda^*]-t)}{s_n+ts_n/(E_{\theta_0=0}[\ell_{1,w^*}(X) \mid \ell_{1,w^*}(X)<\lambda^*]-t) } \\ & = t/{E_{\theta_0=0}[\ell_{1,w^*}(X) \mid \ell_{1,w^*}(X)<\lambda^*]} \\ 
	& \approx t\brackets[\big]{1 + \brackets{1- E_{\theta_0=0}[\ell_{1,w^*}(X) \mid \ell_{1,w^*}(X) <\lambda^*]}},
\end{align*}
with the last approximation following from a Taylor expansion. Finally, one notes (\cref{lem:ExistsLambda+-}) that $E_{\theta_0=0}[\ell_{1,w^*}(X)\mid \ell_{1,w^*}(X)<\lambda^*]$ converges to 1 (from below) at a rate $\eps_n=\log \log (n/s_n)/\log (n/s_n)$. The errors arising each time $\approx$ is invoked above depend on the sparsity $s_n$ and on the boundary separation sequence $v_n$, and are shown in the setting of \cref{thm:FDRConvergenceWithRate} to be of smaller order than $\eps_n$, so that this concludes the (sketch) proof. 
\section{Discussion}



\subsection{Optimality}\label{sec:optimality}

{\em Sharpness of the boundary condition.}
A related work by the current authors \cite{ACR21lb} proves the following bound for fixed $b\in\RR$: any procedure $\vphi$ which is `sparsity preserving' (meaning that its number of discoveries at any $\theta_0\in\ell_0(s_n;b)$ is with high probability not too much larger than $s_n$ --- see \cref{sec:sparsity-preserving} for precise definitions) cannot simultaneously satisfy the FDR and FNR bounds  
\[ \limsup_n \sup_{\theta_0\in \ell_0(s_n;b)} \FDR(\vphi;\theta_0) \leq t,\quad \limsup_n \sup_{\theta_0\in\ell_0(s_n;b)} \FNR(\vphi;\theta_0)=0.\]
Indeed, such a procedure cannot satisfy the FNR bound alone.
In \cref{sec:sparsity-preserving} we show that both the $\Cl$- and $q$-value procedures are sparsity preserving. The condition $v_n\to \infty$ required for \cref{thm:FDRConvergenceNoRate} thus cannot be relaxed while still achieving the FNR bound \cref{eqn:FNRTendsTo0}, and correspondingly for \cref{thm:QvalueControlsFDR}; this is not a limitation specifically of the procedures considered here but is true of any ``reasonable'' (i.e.\ sparsity preserving) procedure, including for example Benjamini--Hochberg type procedures (see \cite{ACR21lb}).\\

{\em General signal regimes.} 
As noted above, 
with weaker signals ($v_n\not\to \infty$) it is impossible to achieve both small FDR and vanishing FNR. One could still hope, under such weaker signals, to have an ``honest'' procedure, in the sense that its FDR is controlled at close to the target level $t$. Simulations in \cite{CR18} suggest that this is indeed the case for the $\Cl$-value procedure. When the $\ell$-value procedure makes no discoveries (i.e.\ every $\ell_{i,\hat{w}}(X)$ is larger than the target level $t$) the $\Cl$-value procedure also makes no discoveries, so that the proofs in \cite{CR18} controlling the FDR of the $\ell$-value procedure for very weak signals also apply to the $\Cl$-value procedure. It remains to study ``intermediate'' signals, strong enough that the $\Cl$-value procedure makes some discoveries but weaker than the class $\ell_0(s_n,v_n)$ analysed here.  While the general case seems challenging, the intermediate signal regime where $v_n$ is fixed ($v_n\not \to -\infty$) can be treated using results from \cite{ACR21lb}: see the remark in \cref{sec:sparsity-preserving}. 

\subsection{Relationship between the $\Cl$-, $\ell$- and $q$-value procedures}

The key contribution of this paper is to analyse the $\Cl$-value procedure. This procedure, like the $q$- and $\ell$- value procedures, is in wide use in multiple testing and does not need our advocacy, but let us nevertheless highlight some advantages.

Note that, as originally introduced in \cite{Storey2003}, $q(x;w)$ corresponds to $P_{(\theta,X)}(\theta_i=0 \:|\: |X_i|\geq x)$. 
Hence, the $q$-value $q_{i,w}(X)$ corresponds to the conditional probability $q_{i,w}(X(\omega))= P_{\theta\sim \Pi_w}(\theta_i=0 \:|\: |X_i|\geq |X_i(\omega)|)$. 
Nevertheless, it is not based solely on the posterior $\Pi_w(\cdot \mid X)$ but rather on the joint distribution of $(\theta,X)$: in the conditioning, the event $\abs{X_i}\geq \abs{X_i(\omega)}$ involves measures $X_i$ more extreme than the observed one $X_i(\omega)$.   By contrast, the \Cl-value procedure depends only on the observed event and not on other events that one hypothetically could have observed. 
From a philosophical point of view, it follows that while both procedures adhere to multiple testing principles, the \Cl-value procedure more closely aligns with Bayesian principles. This potentially also has positive implications for computation, since the $\Cl$-value procedure can be calculated directly from $\ell$-values, while computation of $q$-values must be done separately, requiring an extra integration step, and can be more involved for more complicated priors/models. See \cite{RRV2019} for an example of $q$-value computations for Gaussian mixtures.

When the data is truly generated from the prior, 
$\ell$-value procedures, though optimal for classification problems, are less adapted to the (B)FDR scale than $q$- and $\Cl$- value procedures. Indeed, when the prior $\Pi_w$ correctly specifies the data distribution for some known $w$, these latter procedures achieve $\BFDR$ control at close to the target level, while the $\ell$-value procedure typically has noticeably smaller BFDR (recall also the discussion at the end of \cref{sec:Bayesian-multiple-testing-procedures}, and Proposition 1 in \cite{CR18}).
Similarly, from the frequentist point of view, the results herein and in \cite{CR18} show that in strong signal settings with a (non-random) sparse parameter $\theta_0$, the $\Cl$- and $q$- value procedures make full use of their ``budget'' of false discoveries in order to make more true discoveries, 
while the $\ell$-value procedure undershoots the user-specified target FDR level and so is conservative.

Another approach to adjusting the $\ell$-value procedure to the FDR scale would be to use a \emph{deterministic} threshold $\lambda^*=\lambda^*_n(t)\to 1$ to obtain the specified FDR level $t$ asymptotically. In view of \cref{lem:ExistsLambda+-}, the appropriate choice would have $1-\lambda^*$ of order $(\log (n/s_n))^{-1}$, which depends on the unknown sparsity $s_n$. The $\Cl$-value procedure can be seen as one way to make an appropriate choice adaptively to 
$s_n$.


\subsection{Relationship to frequentist-Bayes analysis} \label{sec:twom}

{\em Frequentist analysis of $\Pi_{\hat w}[\cdot\given X]$.} Recently, a number of works have analysed different aspects of inference for the EB-posterior distribution, mostly from the {\em estimation}  perspective. The paper \cite{JS04} pioneered this study by establishing that the posterior median and mean converge at minimax rates over sparse classes for the quadratic risk. The posterior distribution itself was studied in \cite{cm18} and results on frequentist coverage of Bayesian credible sets were obtained in \cite{cs20}.  This connects to the analysis of Bayesian methods in  high-dimensional settings, where a variety of prior distributions (e.g. different types of  spike-and-slab priors, continuous shrinkage priors including the horseshoe or mixture of Gaussians) and methods (e.g. empirical Bayes, fully Bayes, variational Bayes) have been considered. We refer to \cite{bcg20} for a review on the rapidly growing literature on the subject.\\

{\em Frequentist analysis of Bayesian multiple testing procedures.}

To our knowledge, the only references studying theoretically the frequentist FDR/FNR of Bayesian multiple testing procedures are \cite{CR18} for the present spike-and-slab prior; \cite{Salomond17}, which considers continuous shrinkage priors and derived a first frequentist FDR bound; and \cite{belitser21}, which derived some robust results for model selection based procedures, including for various notions of sums FDR+FNR. 
Let us summarise what was proved in \cite{CR18} and compare with \cref{thm:FDRConvergenceNoRate,thm:FDRConvergenceWithRate,thm:QvalueControlsFDR}: 
\begin{itemize}
	\item the $\ell$-value procedure controls the FDR, uniformly over all sparse alternatives. Its FDR converges to $0$. A logarithmic upper bound is proved for the rate of convergence, but there is no matching lower bound showing that this rate is not bettered. 
	For alternatives with large enough ``signal strength'', the $\ell$-value procedure has a vanishing FNR.
	\item the $q$-value procedure controls the FDR close to the target level, uniformly over all sparse alternatives. For alternatives with large enough signal strength, the $q$-value procedure has FDR converging to the target level, and a vanishing FNR.
\end{itemize}
The commonly used \Cl-value procedure, considered here, was left aside. Its theoretical study is more involved, because it is ``doubly empirical'', with random choices of both $\hat{w}$ and $\hat{\lambda}$.

%


Another key novelty, in addition to considering the \Cl-value procedure, is the weakening of conditions on $v_n$ and on $s_n$. In \cite{CR18} it is assumed that there exists some $\nu<1$ for which $s_n\leq n^\nu$, but we are able to prove \cref{thm:FDRConvergenceNoRate} without this `polynomial sparsity' condition. 
The boundary assumption of \cite{CR18} is equivalent to granting that $v_n\geq b(\log (n/s_n))^{1/2}$ for $b>0$, whereas here we assume only that $v_n\to \infty$. This new condition is sharp: see the discussion in \cref{sec:optimality}.

\subsection{Relationship to latent variables settings} \label{sec:latent-variables}

A model often considered in the literature on multiple testing is the following:
\begin{align} 
	\te=(\te_1,\ldots,\te_n) &\:\: \sim Q\,;\label{mod:latent}\\
	X_i \given \te_i & \stackrel{\mbox{{\tiny indep.}}}{\sim} g_{\te_i} \,,\label{mod:latent2}
\end{align} 
where the $\te_i$'s are random latent states, say taking values in $\{0,1\}$, $Q$ is a probability distribution on such states, and $g_{\te_i}$ is the density of the data point $X_i$ given one is in the state $\te_i$. When the $\te_i$'s are independent, one recovers the so-called two-group model \cite{ETST2001}. Another setting of interest is the case where $Q$ follows a Markov chain, in which case the model  \eqref{mod:latent}--\eqref{mod:latent2} is a Hidden Markov Model (HMM). The work \cite{SC2009} derived results for the \Cl-value multiple testing procedure in the case of parametric assumptions on the emission densities of the HMM, while the nonparametric setting for emission densities has recently been considered in \cite{AGC20}. Other examples include two-sample multiple testing \cite{SC2019} and graph data with underlying stochastic block-model structures \cite{RRV2019}. Such  latent variable approaches can be interpreted as Bayesian methods if we consider the layer \eqref{mod:latent} as a prior distribution. The FDR control provided in {those}  works is {thus} a BFDR control {in the terminology used in the current paper:} that is, an FDR control integrated over the prior, as in \eqref{eqn:def:BFDR}. Said differently, the prior distribution is considered to be ``true'', and the main challenge of these studies is to deal with the estimation of the (hyper-)parameters $Q$ and e.g.\ $g_0,g_1$. 

By contrast, in the sparse setting considered here, we are able to control the FDR without assuming the latent structure \eqref{mod:latent}--\eqref{mod:latent2} is genuinely true.
Results in the two settings are complementary, since uniform guarantees demonstrate the ``robustness'' of the Bayesian approach. However, in the current setting it is essential to choose an \emph{uninformative} prior, hence the heavy (Cauchy) tails of the slab distribution, while in the latent variable setting one must use a correctly specified ``prior'' to obtain optimal results.
Relatedly, sparsity is critical for the current approach  so that the influence of the (fixed, and arbitrary apart from the strong signal assumption) alternatives is not too great. In contrast, 
in a latent variables setting one typically has dense signal, and density is moreover helpful in such a setting since it allows accurate estimation of the distribution of the data under the alternative. As noted above, 
 one success of the current work is to remove the need for polynomial sparsity: this eliminates a gap between the two approaches, allowing our current theorems to work right up to border cases of near density.

\subsection{Possible future research avenues}

This work leaves open several interesting issues. First, extending our results to hierarchical Bayes is both interesting and challenging. Indeed, it has been shown for instance that, for quadratic risk and in sparsity settings, empirical Bayes and hierarchical Bayes posteriors can have different behaviours, even for common choices of spike-and-slab priors (see \cite{cm18}).  One possible route for proving versions of our results for hierarchical priors would be to show that the posterior weight for the hyperparameter $w$ is well enough concentrated. 

Second, we suspect that the C$\ell$-value procedure satisfies some more optimality properties, e.g., having a vanishing FDR+FNR risk tending to zero at the optimal rate provided in \cite{RRJW2020}, if the level $t=t_n$ is chosen to vanish at a suitable rate.

Finally, a probably very challenging issue would be to provide non-asymptotic frequentist FDR controlling results for the C$\ell$-value procedure, and more generally for Bayesian multiple testing procedures. 

\section{Proofs of the main results} \label{sec:proofs}
Throughout the proofs we use the following notation: for a real sequence $(a_n)_{n\in\NN}$ and a non-negative sequence $(b_n)_{n\in \NN}$, we write
$a_n\lesssim b_n$, $b_n\gtrsim a_n$ or $a_n=O(b_n)$ if there exists a constant $C$ such that $\abs{a_n}\leq Cb_n$ for all $n$ large enough; we write $a_n\asymp b_n$ if $a_n\lesssim b_n$ and $b_n\lesssim a_n$; we write $a_n \ll b_n$ or $a_n=o(b_n)$ if $a_n/b_n\to 0$ as $n\to \infty$; and we write $a_n\sim b_n$ if $a_n/b_n\to 1$ as $n\to \infty$. We may also write, for example, $f(w)\sim g(w)$ as $w\to 0$ if $(f/g)(w)\to 1$, and correspondingly.

\subsection{Preliminaries}

To make the sketch argument of \cref{sec:SketchProof} rigorous, we define precise upper and lower bounds $w_\pm$ and $\lambda_\pm$ in place of the central quantities $w^*$, $\lambda^*$. 
There are four parameters governing convergence rates throughout the proof. For a constant $\alpha>0$ to be chosen, (in \cref{lem:HatWConcentrates,lem:lvalsSmallForTrueSignals}),
we write
\begin{align}
\label{eqn:def:nu}		\nu_n &= \alpha s_n^{-1/2}(\log s_n)^{1/2}, \\
\label{eqn:def:delta}	\delta_n&= (\log (n/s_n))^{-1}, \\	
\label{eqn:def:eps}	\eps_n &= \delta_n \log \log (n/s_n),\\
\label{eqn:def:rho} \rho_n &= e^{-v_n^2/9}.
\end{align}
[Recall that the `strong signal assumption' of \cref{thm:FDRConvergenceNoRate} is that $\theta_0\in\ell_0(s_n,v_n)$.]
Note that
\begin{equation}\label{eqn:delta<eps} \delta_n=o(\eps_n).
	\end{equation}
In the setting of \cref{thm:FDRConvergenceWithRate} we further have
	\begin{align} \label{eqn:nu<delta} \nu_n&=o(\delta_n), \\
		\label{eqn:p<delta} \rho_n&\leq \delta_n,
	\end{align}
 the former following from the fact that $u\mapsto (u/\log u)^{-1/2}$ is decreasing on $u>e$ and the assumption that $s_n\geq (\log n)^3$, and the latter from the assumption that $v_n\geq 3 (\log \log (n/s_n))^{1/2}$. 

Define \begin{equation}\label{eqn:def:beta}\beta(x):=\tfrac{g}{\phi}(x)-1,\end{equation} 
and observe we may write the score $S(w)$ as
\begin{equation}\label{eqn:def:Score}
	S(w): = L'(w)=\sum_{i=1}^n \frac{\beta(X_i)}{1+w\beta(X_i)}.
\end{equation} 
Then defining $\tilde{m},m_1$ as in \cite{CR18} by
\begin{align}\label{eqn:def:tildem} \tilde{m}(w)=  -E_{\theta_0=0} \sqbrackets[\Big]{\frac{\beta(X_1)}{1+w\beta(X_1)}} \\ \label{eqn:def:m_1}
	m_1(\tau,w)=E_{\theta_{0,1}=\tau} \sqbrackets[\Big]{ \frac{\beta(X_1)}{1+w\beta(X_1)}},
\end{align}
 we let $w_{\pm}$ be the (almost surely unique) solutions to 
\begin{align}\label{eqn:def:w-} \sum_{i\in S_0} m_1(\theta_{0,i},w_-)=(1+\nu_n)(n-s_n)\tilde{m}(w_-), \\
	\label{eqn:def:w+} \sum_{i\in S_0} m_1(\theta_{0,i},w_+)=(1-\nu_n)(n-s_n)\tilde{m}(w_+).
\end{align}
Note that equations solved by $w_+,w_-$ are close to the expected score equation $E_{\theta_0}[S(w)]=0$. 
While it is shown in \cite{CR18} that solutions exist for $\nu_n=\nu$ a fixed positive constant, strengthening this conclusion to allow $\nu_n\to 0$ is required here to obtain rates of convergence; we note that there indeed exist solutions $w_-\leq  w_+$ to \cref{eqn:def:w+,eqn:def:w-} for $n$ large enough, for any $\alpha>0$, by \cref{lem:ExistsW+-}.

Let \begin{equation}\label{eqn:def:Fw}F_w(x)=P_{\theta_0=0}(\ell_{1,w}\leq x),\end{equation} and for some $A=A(t)>0$ to be chosen (in \cref{lem:HatLambdaConcentrates}), define $\lambda_\pm$ as the solutions to
\begin{align}
	\label{eqn:def:lambda+}
	(n-s_n)F_{w_-}(\lambda_+) \brackets[\big]{E_{\theta_0=0}[\ell_{1,w_+}\mid \ell_{1,w_-}<\lambda_+]-t}&=ts_n + A  s_n \nu_n \\
	\label{eqn:def:lambda-}
	(n-s_n)F_{w_+}(\lambda_-) \brackets[\big]{E_{\theta_0=0}[\ell_{1,w_-} \mid \ell_{1,w_+}<\lambda_-]-t}&=ts_n - As_n\max(\nu_n,\rho_n,\delta_n).
\end{align}
Note that unique solutions $\lambda_-<\lambda_+$ to \cref{eqn:def:lambda+,eqn:def:lambda-} exist by \cref{lem:ExistsLambda+-}. These definitions correspond to the central approximation $(n-s_n)F_{w^*}(\lambda^*)(E_{\theta_0=0}[\ell_{1,w^*} \mid \ell_{1,w^*}<\lambda^*]-t)=ts_n$, but with each central quantity $w^*$ replaced by an upper or lower bound in a consistent way to ensure the left side is always decreased in \cref{eqn:def:lambda+} and always increased in \cref{eqn:def:lambda-}, so that the solutions $\lambda_+$ and $\lambda_-$ will indeed bound $\hat{\lambda}$ from above and below respectively (\cref{lem:HatLambdaConcentrates}).

\subsection{Proof of Theorems~\ref{thm:FDRConvergenceNoRate} and \ref{thm:FDRConvergenceWithRate}} \label{sec-proof-main}
\Cref{sec-core} will provide a number of core lemmas which allow a concise exposition of the proofs of \cref{thm:FDRConvergenceWithRate,thm:FDRConvergenceNoRate}.
In particular, \cref{lem:HatWConcentrates,lem:lvalsSmallForTrueSignals,lem:HatLambdaConcentrates,lem:VLambdaWConcentrates} collectively tell us, via a union bound, that there exists an event $\Aa$ of probability at least $1-\nu_n$ on which, for some $a=a(t)>0$ and with 	$K_n:=\#\braces{i\in S_0 : \ell_{i,w_-}>\delta_n}$,
	\begin{equation} \label{eqn:eventAa} \begin{split} \hat{w}&\in(w_-,w_+),\\ 
			K_n&\leq s_n (\rho_n+\nu_n), \\ 
			\hat{\lambda}&\in [\lambda_-,\lambda_+],\\
			V_{\lambda_+,w_+}&\leq E[V_{\lambda_+,w_+}]+as_n\nu_n, \\
			V_{\lambda_-,w_-}&\geq E[V_{\lambda_-,w_-}]-a s_n\nu_n.\end{split}
	\end{equation}

	{\em FNR control.} 	By monotonicity of the $\ell$-values (\cref{lem:monotonicity}) and the fact that $\lambda_-$ is bounded away from zero (as implied by \cref{lem:ExistsLambda+-}) we note that for $n$ large we have on $\Aa$ 
	\begin{equation} \label{eqn:KnBoundsMissedSignals}\#\braces{i\in S_0 : \vphi^{\Cl}_i=0} \leq \#\braces{i \in S_0 : \ell_{i,w_-}\geq \lambda_-}  \leq K_n,\end{equation} 
	which in particular allows us to immediately deduce the FNR control \cref{eqn:FNRTendsTo0}:
	\begin{equation}\label{eqn:FNRproof}\FNR(\vphi^{\Cl};\theta_0)\leq E_{\theta_0}\brackets[\Big]{\frac{K_n}{s_n}\II_{\Aa}+\II_{\Aa^c}}\leq \rho_n+\nu_n + P_{\theta_0}(\Aa^c)\leq \rho_n+2\nu_n\to 0. \end{equation}
	In the setting of \cref{thm:FDRConvergenceWithRate}, the fact that $\max(\rho_n,\nu_n)\leq \delta_n$ (recall \cref{eqn:nu<delta,eqn:p<delta}) implies the FNR claim \cref{eqn:FNRControl}.
	
{\em FDR upper bound.}
	We turn now to the control of the false discovery rate. By monotonicity (see \cref{lem:monotonicity}), on the event $\Aa$, the number $V_{\hat{\lambda},\hat{w}}$ of false discoveries made by $\vphi^\Cl$ lies between $V_{\lambda_-,w_-}$ and $V_{\lambda_+,w_+}$.
	By \cref{lem:ExistsLambda+-} we see for a constant $D=D(t)>0$ that \[E[V_{\lambda_+,w_-}]=(n-s_n)F_{w_-}(\lambda_+)\leq (1+D\max(\eps_n,\nu_n))t(1-t)^{-1}s_n.\] \Cref{lem:EVw+closetoEVw-}  tells us that $E[V_{\lambda_+,w_+}]\leq (1+B\max(\nu_n,\rho_n,\delta_n))E[V_{\lambda_+,w_-}]$ for some constant $B$, hence for some constant $D'>D$ we deduce using $\delta_n=o(\eps_n)$ that for $n$ large enough we have
	\[E[V_{\lambda_+,w_+}]\leq (1+D' \max(\eps_n,\nu_n,\rho_n))t(1-t)^{-1}s_n.\]
		Since for $a,b>0$, the map $x\mapsto x/(a+x)$ is increasing and the map $b/(a+x)$ is decreasing on $x> -a$, using also  \cref{eqn:eventAa,eqn:KnBoundsMissedSignals}  we deduce that on the event $\Aa$
	\begin{align*}
		\FDP(\vphi^{\Cl}
		;\theta_0) & \leq \frac{V_{\hat{\lambda},\hat{w}}}{V_{\hat{\lambda},\hat{w}}+s_n-K_n}\\
		&\leq \frac{ E[V_{\lambda_+,w_+}]+a s_n\nu_n}{E[V_{\lambda_+,w_+}]+a s_n\nu_n +s_n-K_n} \\ &\leq \frac{(1+D'\max(\eps_n,\nu_n,\rho_n))t(1-t)^{-1}s_n+a s_n\nu_n
		}{s_n[(1+D'\max(\eps_n,\nu_n,\rho_n))t(1-t)^{-1}+1-(\rho_n+(1-a)\nu_n)]}
		\\ & \leq \frac{t+D't\eps_n+a'\max(\nu_n,\rho_n) }{1+D't\eps_n -a'\max(\nu_n,\rho_n)},
		\end{align*}
		for some $a'=a'(t)>0$, hence we have \[\FDP(\vphi^{\Cl};\theta_0)\leq \frac{t+D't\eps_n+a'\max(\nu_n,\rho_n) }{1+D't\eps_n -a'\max(\nu_n,\rho_n)} + \II_{\Aa^c}.\]
Taking expectations, using that $P_{\theta_0}(\Aa^c)\leq \nu_n$ and that
\[ \frac{t+D't\eps_n}{1+D't\eps_n} = t +  \frac{D't(1-t)\eps_n} {1+D't\eps_n}\leq t+D't(1-t)\eps_n,\]
by Taylor expanding we see that for some constant $A'=A'(t)$, for $n$ large we have
\[ \FDR(\vphi^{\Cl}
;\theta_0)\leq t+ t(1-t) D'\eps_n +A'\max(\nu_n,\rho_n).\]
The right side converges to $t$ in the settings of \cref{thm:FDRConvergenceNoRate,thm:FDRConvergenceWithRate}. In the latter setting we note $\max(\nu_n,\rho_n)=o(\eps_n)$ by \cref{eqn:delta<eps,eqn:nu<delta,eqn:p<delta}, and the upper bound in \cref{eqn:FDRUpperAndLowerBound} follows.

{\em FDR lower bound.}	
	For the lower bound, note by \cref{lem:ExistsLambda+-} that for a constant $d>0$ we have \[E[V_{\lambda_-,w_+}]\geq\frac{t}{1-t}s_n \brackets[\big]{ 1+d \eps_n - \tfrac{A}{t}\max(\nu_n,\rho_n)},\] for $n$ large. Thus, by \cref{lem:EVw+closetoEVw-} and for $B$ the constant thereof, using that $\delta_n=o(\eps_n)$ we see that for some constants $A',d'$ depending on $t$ and for $n$ larger than some $N=N(t)$ we have 
	\[\begin{split} E[V_{\lambda_-,w_-}] &\geq (1-B\max(\nu_n,\rho_n,\delta_n))(1+d\eps_n- \tfrac{A}{t}\max(\nu_n,\rho_n))t(1-t)^{-1}s_n\\ &\geq (1+d'\eps_n- A'\max(\nu_n,\rho_n))t(1-t)^{-1}s_n,\end{split}\]
	hence, using \cref{eqn:eventAa} and upper bounding the number of true discoveries by $s_n$,
	\begin{align*}
		\FDP(\vphi^{\Cl}
		;\theta_0) & \geq \frac{ E[V_{\lambda_-,w_-}]- a s_n\nu_n}{s_n+E[V_{\lambda_-,w_-}]- a s_n \nu_n}\II_{\Aa} \\
		&\geq \frac{(1+d'\eps_n-A'\max(\nu_n,\rho_n))t(1-t)^{-1}s_n- a s_n \nu_n}
		{s_n+(1+d'\eps_n-A'\max(\nu_n,\rho_n))t(1-t)^{-1}s_n-a s_n\nu_n}-\II_{\Aa^c} \\ 
		&\geq \frac{t+d't\eps_n-a'\max(\nu_n,\rho_n)}{1+d't\eps_n-a'\max(\nu_n,\rho_n)} -\II_{\Aa^c},
	\end{align*}
for $a'=A't+a(1-t)$.
Similarly to the upper bound we note that for large $n$
\[\frac{t+d't\eps_n}{1+d't\eps_n}=t+\frac{d't(1-t)\eps_n}{1+d't\eps_n}\geq t+0.5t(1-t)d'\eps_n,\] so that Taylor expanding and taking expectations, recalling that $P_{\theta_0}(\Aa^c)\leq \nu_n$, we obtain for some $A''=A''(t)$	\[ \FDR(\vphi^{\Cl}
;\theta_0)\geq t + 0.5t(1-t) d'\eps_n - A''\max(\nu_n,\rho_n).\]
Again the right side tends to $t$ in the settings of both \cref{thm:FDRConvergenceNoRate,thm:FDRConvergenceWithRate}. In the latter setting, for all $n$ greater than some $N=N(t)$, we have $0.5t(1-t)d'\eps_n>2A''\max(\nu_n,\rho_n)$, and the lower bound in \cref{eqn:FDRUpperAndLowerBound} follows. 

\subsection{Proof of Theorem~\ref{thm:QvalueControlsFDR}}\label{sec:proofs:qval}
Let us prove \cref{thm:QvalueControlsFDR} in the setting of \cref{thm:FDRConvergenceWithRate}; the proof under the weaker conditions of \cref{thm:FDRConvergenceNoRate} is similar and omitted.  
As with the proof of \cref{thm:FDRConvergenceNoRate,thm:FDRConvergenceWithRate}, by \cref{lem:HatWConcentrates,lem:lvalsSmallForTrueSignals} there exists an event $\Aa$ of probability at least $1-\nu_n$ on which, for $K_n:=\#\braces{i\in S_0 : \ell_{i,w_-}>\delta_n}$,
\begin{equation*} \begin{split} \hat{w}&\in(w_-,w_+),\\ 
		K_n&\leq s_n (\rho_n+\nu_n).
	\end{split}
\end{equation*}
 By monotonicity of the $q$-values (\cref{lem:monotonicity}) it will be enough to consider the tests $\brackets{\II\braces{q_{i,w}<t}}_{1\leq i\leq n}$ for $w=w_-,w_+$.

{\em First step: control of false negatives.}
Define \[ S'_w=\#\braces{i\in S_0 : q_{i,w}<t},\]
so that  $\FNR(\vphi^\qval;\theta_0)=s_n^{-1}E_{\theta_0}[s_n-S'_{\hat{w}}]$.

In view of the fact that $\ell_{i,w_-}\geq q_{i,w_-}$ (see \cref{lem:ResultsFromCR18}) we note that for $n$ large we have
\[
S'_{w_-}=\sum_{i\in S_0} \II\braces{q_{i,w_-}<t} \geq  \sum_{i\in S_0} \II\braces{\ell_{i,w_-}< t}\geq s_n - K_n ,
\]
so that on $\Aa$, using monotonicity of $q$-values,
\[ S_{\hat{w}}'\geq S_{w_-}' \geq s_n(1-\nu_n-\rho_n).\]
In the current setting $\max(\nu_n,\rho_n)\leq \delta_n$, so that
\begin{equation*}\FNR(\vphi^{\qval};\theta_0)
	=s_n^{-1} E_{\theta_0}[s_n-S_{\hat{w}}'] \leq  	 E_{\theta_0}\brackets[\big]{(\nu_n+\rho_n)\II_{\Aa}+\II_{\Aa^c}}
	\leq \rho_n + 2\nu_n
	\leq 3\delta_n, \end{equation*}
proving \cref{eqn:FNRControl} for the $q$-value procedure.

We proceed with the proof of the FDR lower bound. As with the proofs in the $\Cl$-value case, the key remaining steps are to prove the concentration of, and to control the expectation of, the number of false positives, and we begin with the latter.

{\em Second step: bounding the expected number of false positives.}
Define functions $r:(0,1)^2\to [0,\infty)$ and $\chi : (0,1]\to [0,\infty)$ by
\begin{align}\label{eqn:def:r}
	r(w,t)&= \frac{wt}{(1-w)(1-t)},\\
	\label{eqn:def:chi}
	\chi(x)&= (\bar{\Phi}/\bar{G})^{-1}(x).
\end{align}
Note that $\chi$ is well-defined and strictly decreasing because $\bar{\Phi}/\bar{G}$ itself is strictly decreasing on $[0,\infty)$ (see \cref{lem:monotonicity}). Moreover, recalling the definition \cref{eqn:def:qvals} of the $q$-values, we note that for any $w\in[0,1)$ and $t\in [0,1)$, 
\[\braces{q_{i,w}< t}=\braces{\abs{X_i}> \chi(r(w,t))}.\]
We write
\[
V'_{w}=\sum_{i\notin S_0} \II\braces{q_{i,w}< t} = \sum_{i\notin S_0} \II\braces{|X_i|> \chi(r(w,t))}
\]
for the number of false positives of the multiple testing procedure $\brackets{\II\braces{q_{i,w}<t}}_{1\leq i\leq n}$.
 Note that $V'_{w}$ is increasing in $w\in (0,1)$ (\cref{lem:monotonicity}) and by definition of $\chi$ satisfies \begin{equation}\label{eqn:EV'}E_{\theta_0} V'_{w}=2(n-s_n)\overline{\Phi}(\chi(r(w,t)))= (n-s_n)r(w,t)2\overline{G}(\chi(r(w,t))),\end{equation} provided $r(w,t)\leq 1$. 
From \cref{lemmaBoundingqvalue}, we have
\begin{align}
	\tilde{m}(w)\left(1+ c \frac{\log\log (1/w)}{\log (1/w)}\right) \leq 2 \overline{G}(\chi(r(w,t))) \leq \tilde{m}(w)\left(1+ c' \frac{\log\log (1/w)}{\log (1/w)}\right)
	\label{Boundingqvalue}
\end{align} 
for $w$ small enough (smaller than some threshold possibly depending on $t$). 
Using the definition \eqref{eqn:def:w-} of $w_-$ to translate from $\tilde{m}$ to $m_1$, \cref{lem:m1boundsAnyBoundary} to lower bound $m_1(\theta_{0,i},w_-)$, and that $w_-\asymp (s_n/n) (\log(n/s_n))^{1/2}$ (which implies also that $\log\log(1/w_-)/\log(1/w_-)\asymp \log\log(n/s_n)/\log(n/s_n)=\eps_n$) by \cref{lem:ExistsW+-}, we obtain
\begin{align*}
 E_{\theta_0} V'_{w_-} &\geq  (n-s_n)\frac{w_-}{1-w_-} \frac{t}{1-t}\tilde{m}(w_-)\left(1+ c \frac{\log\log (1/w_-)}{\log (1/w_-)}\right)\\
 &\geq \frac{t}{1-t} w_-\sum_{i\in S_0} m_1(\theta_{0,i},w_-)(1+\nu_n)^{-1} \brackets*{1+ c \eps_n}\\
	&\geq  s_n \frac{t}{1-t} (1+\nu_n)^{-1}  (1-\rho_n) \left(1+ c_1 \eps_n\right)\\
	&\geq s_n \frac{t}{1-t}  \left(1+ c_2 \eps_n\right),
\end{align*}
for some constants $c_1,c_2>0$, because $(1+\nu_n)^{-1}=1-O(\nu_n)$ and $\max(\nu_n,\rho_n)=o(\eps_n)$.

{\em Third step: concentration of the number of false positives.}
Recalling that $\nu_n=\alpha s_n^{-1/2}(\log s_n)^{1/2}$, we see from an application of Bernstein's inequality (\cref{lem:Bernstein}) and the above that for some constant $c_3>0$ and for $a=a(t)$ large enough,
\begin{multline*}
	P_{\theta_0}(V'_{w_-} - E_{\theta_0}V'_{w_-} \geq -a\nu_n E_{\theta_0}V'_{w_-}) 
	\leq \exp\{-(3/8) a^2  \nu_n^2 E_{\theta_0}V'_{w_-}\} \\ \leq e^{-c_3 a^2 s_n \nu_n^2 t/(1-t) } \leq s_n^{-1/2}.
\end{multline*}

{\em Fourth step: deriving the FDR lower bound.}
Using the previous steps and upper bounding the number of true positives by $s_n$, we obtain by using again the monotonicity of the $q$-values and that the map $x\mapsto x/(a+x)$ is increasing on $x>-a$ that
\begin{align*}
	\FDP(\vphi^{\qval};\theta_0)&\geq \frac{V'_{\hat{w}}}{V'_{\hat{w}}+s_n}\\
	&\geq \frac{V'_{w_-}}{V'_{w_-}+s_n} \II\braces{\hat{w} \geq w_-}\\
	&\geq \frac{(1-a\nu_n) E_{\theta_0} V'_{w_-}}{(1-a\nu_n) E_{\theta_0} V'_{w_-}+s_n} \II\braces{\hat{w} \geq w_-,V'_{w_-} \geq (1-a\nu_n) E_{\theta_0}V'_{w_-}},
\end{align*}
for $a=a(t)$ as above. 
Taking the expectation and using the bounds we have attained on probabilities, we find
\begin{align*}
	\FDR(\vphi^{\qval};\theta_0)&\geq \frac{(1-a\nu_n) E_{\theta_0} V'_{w_-}/s_n}{(1-a\nu_n) E_{\theta_0} V'_{w_-}/s_n+1}  - 2\nu_n.
\end{align*}
Using the previously obtained bound on $E_{\theta_0}V_{w_-}'$ and the fact that $(1-a\nu_n)(1+c_2\eps_n)\geq 1+c\eps_n$ for some $c>0$ and $n$ large enough, we find that
\begin{align*}
	\frac{(1-a\nu_n) E_{\theta_0} V'_{w_-}/s_n}{(1-a\nu_n) E_{\theta_0} V'_{w_-}/s_n+1}
	\geq & \frac{(1+c\eps_n)t}{(1+c\eps_n)t+1-t} = \frac{t+c\eps_nt}{1+c\eps_nt} \\ &= t + \frac{c\eps_nt(1-t)}{1+c\eps_nt} \geq t+ 0.5 t(1-t) c\eps_n,
\end{align*}
and we deduce the FDR lower bound.

{\em Fifth step: deriving the FDR upper bound.}
Recall that on the event $\Aa$, for $n$ large we have both $S'_{w_-}\geq  s_n (1-\nu_n-\rho_n) $ and $w_-\leq \hat{w}\leq w_+$. Again using that $x\mapsto x/(a+x)$ is increasing and here also that $x\mapsto b/(a+x)$ is decreasing on $x>-a$,
\begin{align*}
	\FDP(\vphi^{\qval};\theta_0)&\leq \frac{V'_{w_+}}{(V'_{w_+}+S'_{w_-})\vee 1} \II_\Aa + \II_{\Aa^c} \\
	&\leq \frac{V'_{w_+}}{V'_{w_+}+s_n (1-\nu_n-\rho_n)} \II_\Aa + \II_{\Aa^c} .
\end{align*}
Here one could use a concentration argument as for the lower bound, but noting that $x\mapsto x/(a+x)$ is concave, we bypass the need for this by appealing to Jensen's inequality to obtain
\begin{align}
	\FDR(\theta_0,\vphi^{\qval};\theta_0)
	&\leq \frac{E_{\theta_0} V'_{w_+}}{E_{\theta_0} V'_{w_+}+s_n (1-\nu_n-\rho_n)}  + \nu_n.\label{equintermupqval}
\end{align}
For upper bounding $E_{\theta_0} V'_{w_+}$, we proceed as for the lower bound part: using \eqref{Boundingqvalue}, the definition \eqref{eqn:def:w+} of $w_+$, \cref{lem:m1boundsAnyBoundary}, and that $w_+\asymp (s_n/n)(\log(n/s_n))^{1/2}$ (so that $\log\log(1/w_+)/\log(1/w_+)\asymp \eps_n$ and $w_+=o(\eps_n)$) by \cref{lem:ExistsW+-}, we find
\begin{align*}
	E_{\theta_0} V'_{w_+}&\leq (n-s_n)r(w_+,t) \tilde{m}(w_+)\left(1+ c'\eps_n\right)\\
	&\leq t(1-t)^{-1}(1-w_+)^{-1} s_n (1-\nu_n)^{-1} \left(1+ c' \eps_n\right)\\
	&\leq  t(1-t)^{-1} s_n (1+c\eps_n),
\end{align*}
for any $c>c'$, for $n$ larger than some $N(t)$, using again that $\nu_n=o(\eps_n)$. Substituting into \eqref{equintermupqval} and recalling that we also have $\rho_n=o(\eps_n)$ yields
\begin{align*}
	\FDR(\vphi^{\qval};\theta_0)
	&\leq \frac{t (1+c\eps_n)}{t (1+c\eps_n)+ (1-t)(1-\nu_n-\rho_n)}  + \nu_n\\
	&\leq \frac{t+tc\eps_n}{1+tc\eps_n}  + o(\eps_n)\\
	&\leq t +t(1-t)c\eps_n +o(\eps_n).
\end{align*}
This completes the upper bound and hence the proof.

\section{Core lemmas}\label{sec-core}
\subsection{Statements}\label{sec-core-statements}
The following monotonicity results are mostly clear from the definitions.
\begin{lemma}[Monotonicity]\label{lem:monotonicity}
	We have the following monotonicity results, all of which may be non-strict unless specified. 
	
	As $w\in(0,1)$ increases, with other parameters fixed (note that we typically apply these results with $n$ increasing and $w=w_n$ decreasing),
	\begin{align*}
	1=\ell_{i,0}(X)\geq \ell_{i,w}(X) &\downarrow 0 \qquad \text{(strictly)} \\
	1=q_{i,0}(X)\geq q_{i,w}(X)&\downarrow 0 \qquad \text{(strictly)} \\
	V_{\lambda,w}&\uparrow (n-s_n),~\lambda\in (0,1] 	\\
	 V'_{w}&\uparrow (n-s_n)
		 \\
	\postFDR_w(\vphi)&\downarrow 0 		\\
		\postFDR_u(\vphi_{\lambda,w})&\uparrow \frac{1}{n}\sum_{i=1}^n \ell_{i,u},~u\in(0,1) 
		 \\
		L'(w)=S(w)&\downarrow \textstyle\sum_{i=1}^n\tfrac{\beta(X_i)}{1+\beta(X_i)} \quad \text{(a.s.\ strictly).} 
\intertext{	For fixed $w,w'\in(0,1)$, as $\lambda\in [0,1]$ increases, }
		V_{\lambda,w}&\uparrow n-s_n, 
		\\
F_{w}(\lambda)&\uparrow 1 
\qquad \text{(strictly)} \\
	E_{\theta_0=0}[\ell_{1,w} \mid \ell_{1,w'}<\lambda] &\uparrow E_{\theta_0=0}[\ell_{1,w}] 
	\end{align*}
Finally, we note that $(\phi/g)(x)$ and $(\bar{\Phi}/\bar{G})(x)$ decrease strictly as $x\in [0,\infty)$ increases, hence the functions $\ell(x;w)$ and $q(x;w)$ defined in \cref{eqn:def:Lvals,eqn:def:Qfunction} decrease on this set for fixed $w$.
\end{lemma}

The following lemmas then form the core of the proofs of \cref{thm:FDRConvergenceWithRate,thm:FDRConvergenceNoRate}. Some ancillary results used in the proofs of these lemmas are relegated to \cref{sec:AuxiliaryLemmas}.
\begin{lemma}\label{lem:ExistsW+-}
	Under the assumptions of \cref{thm:FDRConvergenceNoRate}, define $\nu_n$ as in \cref{eqn:def:nu} with $\alpha>0$ arbitrary. Then for $n$ large there exist solutions $w_-\leq w_+$ to \cref{eqn:def:w-,eqn:def:w+}. Moreover, these solutions are almost surely unique and satisfy, for $w\in \braces{w_-,w_+}$,
 \[w \asymp s_n (n-s_n)^{-1} \tilde{m}(w)^{-1} \asymp s_n(n-s_n)^{-1}(\log(n/s_n))^{1/2}\asymp (s_n/n) (\log( n/s_n))^{1/2}. \] 
\end{lemma}

\begin{lemma}\label{lem:ExistsLambda+-} 	In the setting of \cref{thm:FDRConvergenceNoRate}, for any constant $A$ there exist unique solutions $\lambda_-<\lambda_+$ to \cref{eqn:def:lambda+,eqn:def:lambda-}, and these solutions satisfy
	\begin{equation}\label{eqn:Lambda+-To1} 1- \lambda_- \asymp 1-\lambda_+ \asymp \delta_n, \end{equation}
	with suppressed constants depending on $t$.
	We further note that for some constants $C,c>0$ depending on $t$,
	\begin{align}\label{eqn:Elambda+to1} E_{\theta_0=0}[\ell_{1,w_+} \mid \ell_{1,w_-} < \lambda_+] &\geq 1- C\eps_n, \\ \label{eqn:Elambda-to1} E_{\theta_0=0}[\ell_{1,w_-}\mid \ell_{1,w_+}<\lambda_-]&\leq 1- c\eps_n,\end{align} 
	and that for some $D,d>0$ depending on $t$, recalling $F_w(\lambda):=P_{\theta_0=0}(\ell_{1,w}<\lambda),$
	\begin{align}
\label{eqn:Flambda+geq}		(n-s_n) F_{w_-}(\lambda_+)&\leq \frac{t}{1-t}s_n\brackets[\big]{ 1+D\max(\eps_n,\nu_n) }, \\
\label{eqn:Flambda-geq}		(n-s_n) F_{w_+}(\lambda_-)&\geq\frac{t}{1-t}s_n \brackets[\big]{ 1+d \eps_n - \tfrac{A}{t} \max(\nu_n,\rho_n)} ,
	\end{align}
for all $n$ large enough.
\end{lemma}

\begin{lemma}\label{lem:HatWConcentrates}
	Under the assumptions of \cref{thm:FDRConvergenceNoRate}, 	recalling the definition \cref{eqn:def:wHat} of $\hat{w}$ and the definitions \cref{eqn:def:w+,eqn:def:w-} of $w_\pm$, we have
	\begin{equation}P_{\theta_0}(\hat{w}\not \in(w_-,w_+)) =o(\nu_n),\end{equation}
	provided the constant $\alpha$ in the definition \cref{eqn:def:nu} of $\nu_n$ is large enough.
\end{lemma}

\begin{lemma}\label{lem:HatLambdaConcentrates}
	Under the assumptions of \cref{thm:FDRConvergenceNoRate}, 	recalling the definition \cref{eqn:def:lambdahat} of $\hat{\lambda}$ as the threshold of $\vphi^{\Cl}$ and the definitions \cref{eqn:def:lambda+,eqn:def:lambda-} of $\lambda_\pm$, 
	we have
	\begin{equation}
		\label{eqn:LambdaHatConcentrates}
		P_{\theta_0}(\hat{\lambda}\not \in [\lambda_-,\lambda_+]) = o(\nu_n),\end{equation}
	provided the constant $A=A(t)$ is large enough in the definitions of $\lambda_\pm$.
\end{lemma}

\begin{lemma}\label{lem:lvalsSmallForTrueSignals}
	In the setting of \cref{thm:FDRConvergenceNoRate}, recall that $S_0$ denotes the support of $\theta_0$ as in \cref{eqn:def:S0} and define the (random) set $S_1=\braces{i \in S_0 : \ell_{i,w_-}\leq \delta_n},$ where $w_-$ is as in \cref{eqn:def:w-}. Then, defining 
	\begin{equation}\label{eqn:def:Kn} K_n=\abs{S_0\setminus S_1} = \#\braces{i\in S_0 : \ell_{i,w_-}> \delta_n},\end{equation} for all $n$ large enough we have
	\begin{equation}P_{\theta_0}\brackets[\big]{K_n/s_n> \rho_n+\nu_n}= o(\nu_n),\end{equation}
	provided the constant $\alpha$ in the definition \cref{eqn:def:nu} of $\nu_n$ is large enough.
\end{lemma}

\begin{lemma}\label{lem:VLambdaWConcentrates}
In the setting of \cref{thm:FDRConvergenceNoRate},  define $V_{\lambda,w}$ as in \cref{eqn:def:VLambdaW}. Then
\begin{align*}
	P_{\theta_0}(\abs{V_{\lambda_+,w_+}- E[V_{\lambda_+,w_+}]}>as_n\nu_n)= o(\nu_n),
\end{align*}	
for some constant $a=a(t)$. The same holds upon replacing one or both of $\lambda_+$ and $w_+$ respectively with $\lambda_-$ and $w_-$.
\end{lemma}

\begin{lemma}\label{lem:EVw+closetoEVw-}
	In the setting of \cref{thm:FDRConvergenceNoRate}, recall the definitions \cref{eqn:def:VLambdaW,eqn:def:w+,eqn:def:w-,eqn:def:lambda+,eqn:def:lambda-} of $V_{\lambda,w},$ $w_\pm,$ and $\lambda_\pm$.
	Then for some constant $B>0$,
	\begin{align} \label{eqn:EV++} EV_{\lambda_+,w_+}&\leq E V_{\lambda_+,w_-}\brackets[\Big]{1+B\max(\nu_n,\rho_n,\delta_n)}, \\ \label{eqn:EV--}
		E V_{\lambda_-,w_-} &\geq E V_{\lambda_-,w_+} \brackets[\Big]{1-B\max(\nu_n,\rho_n,\delta_n)}.\end{align}
\end{lemma}

\subsection{Proofs
}\label{sec-core-proofs}
We here define two final quantities which appear in the proofs, closely related to $\chi$ as defined in \cref{eqn:def:chi}: recalling the definition $\beta(x)=(g/\phi)(x)-1$ from \cref{eqn:def:beta}, we set
\begin{align}
	\label{eqn:def:xi} \xi(x)&= (\phi/g)^{-1}(x),~x\in(0,(\phi/g)(0)]\\
	\label{eqn:def:zeta} \zeta(w)&=\beta^{-1}(1/w),~~ w\in (0,1].
\end{align} 
Note the relationship 
\begin{equation}
	\label{eqn:zeta-xi-relationship}
\zeta(w)=\xi(w/(1+w)).
\end{equation}

\begin{proof}[Proof of \cref{lem:monotonicity}] Strict monotonicity in $w$ of $\ell_{i,w}$, $q_{i,w}$ is immediate from the definitions \cref{eqn:def:Lvals,eqn:def:qvals}: for example,
	\[ \ell_{i,w}= \frac{(1-w)\phi(X_i)}{(1-w)\phi(X_i)+wg(X_i)} = \frac{1}{1+(w/(1-w))(g/\phi)(X_i)}\] decreases as $w$ increases because $(g/\phi)(X_i)> 0$. Non-strict monotonicity in $w$ of $V_{\lambda,w}$, $V'_{w}$, $\postFDR_{w}(\vphi)$ follows immediately. The monotonicity of $V_{\lambda,w}$ in $\lambda$ is also clear (and note that $\ell_{i,w}<1$ for $w\in(0,1)$ so that $(\vphi_{1,w})_i=1$ for all $i$). 	To see that \[ \postFDR_{u}(\vphi_{\lambda,w_2})\geq \postFDR_{u}(\vphi_{\lambda,w_1}) \quad \text{if} \quad w_2\geq w_1,\] note that changing $w$ does not change the ordering of the $\ell_{i,w}$ values, only their magnitudes, since smaller $\ell$ values correspond to larger values of $\abs{X_i}$. It follows that $\ell_{i,w_1}$ and $\ell_{i,w_2}$ both select coordinates $i$ in order of increasing $\ell_{i,u}$ values. Since the $\ell_{i,w}$ are monotonic in $w$, we see that $\vphi_{\lambda,w_2}$ selects every $i$ selected by $\vphi_{\lambda,w_1}$, so that $\postFDR_u(\vphi_{\lambda,w_2})$, which can be viewed as the average of the selected $\ell_{i,u}$ values (cf.\ \cref{eqn:def:HatK}), is no smaller than $\postFDR_u(\vphi_{\lambda,w_1})$.

	Strict decreasingness of $\phi/g$ is immediate from the definition \cref{eqn:def:gInQuasiCauchy},  and implies the same of $\bar{\Phi}/\bar{G}$ (see \cite[Lemma S-9]{CR18}). In view of the explicit expression for $F_w$ in \cref{lem:AsymptoticsOfFwLambda}, its strict monotonicity follows from that of $\bar{\Phi}$, $\xi=(\phi/g)^{-1}$ and $r(w,\lambda)=w\lambda(1-w)^{-1}(1-\lambda)^{-1}$. Similarly the score function $S(w)=L'(w)$ defined in \cref{eqn:def:Score} can be seen, by differentiating, to be strictly decreasing on the event where there exists $i$ such that $\beta(X_i)\neq 0$, which has probability 1 because $\beta(x)=(g/\phi)(x)-1$ is strictly increasing and the $X_i$'s have non-atomic distributions.

	For monotonicity of $E_{\theta_0=0}[\ell_{1,w} \mid \ell_{1,w'} <\lambda]$ in $\lambda$, first note that, writing $\xi_w(\lambda)=\xi(r(w,\lambda))$, a direct calculation yields  
	\[ \braces{\ell_{i,w}<\lambda}=\braces{\abs{X_i}>\xi_w(\lambda)}.\]  It follows that 
	\[ \braces{\ell_{1,w'}<\lambda} = \braces{\ell_{1,w}<\xi_{w}^{-1}\circ \xi_{w'}(\lambda)},\] and hence that we can express the expectation as
	\[ E_{\theta_0=0} [\ell_{1,w} \mid \ell_{1,w'}<\lambda] = Z\circ \xi_{w}^{-1}\circ \xi_{w'}(\lambda),\quad  Z(x)= E_{\theta_0=0} [ \ell_{1,w} \mid \ell_{1,w}<x]. \] 
	 It suffices, since $\xi_{w}$ is decreasing, to note that $Z$ is increasing, which is intuitively clear and formally follows from the following calculations: writing $U=\ell_{1,w}$, for $b>a$ we have
	\begin{multline} E[ U \mid U<b]
	= E[ U \mid  U<a] \Pr(U<a \mid  U<b) \\ + E[ U \mid a\leq U<b] \Pr(U\geq a \mid U<b)\end{multline}
	Then, since  $E[ U \mid  a\leq U<b] \geq a \geq E[U \mid U<a],$ we deduce that
	\begin{align*} E[U\mid U<b] \geq& E[U\mid U<a] \brackets{\Pr(U<a \mid  U<b) + \Pr(U\geq a \mid U<b)} \\ &= E[U\mid  U<a].\qedhere\end{align*}
\end{proof}

\begin{proof}[Proof of \cref{lem:ExistsW+-}]
The essence of the proof is that $m_1(\theta_{0,i},w)\asymp 1/w$ and $\tilde{m}(w)\asymp (\log(1/w))^{-1/2}$, so that necessarily if \cref{eqn:def:w-} or \cref{eqn:def:w+} is satisfied by some $w$ we have $s_n/w \asymp (n-s_n)(\log(1/w))^{-1/2}$ and hence $w\asymp (s_n/n)(\log(n/s_n))^{1/2}$. 

 To make this precise, we begin by claiming that, for some constant $C>0$,
	\begin{align} 
		\label{eqn:sn<w-}\sum_{i \in S_0} m_1(\theta_{0,i},s_n/n)&>(1+ \nu_n)(n-s_n)\tilde{m}(s_n/n) \\
		\label{eqn:snlogsn>w+}\sum_{i \in S_0} m_1\brackets[\big]{\theta_{0,i},C\tfrac{s_n}{n}\brackets[\big]{\log( \tfrac{n}{s_n})}^{1/2}}&<(1- \nu_n)(n-s_n)\tilde{m}\brackets[\big]{C\tfrac{s_n}{n}\brackets[\big]{\log( \tfrac{n}{s_n})}^{1/2}},
	\end{align}
	at least for large enough $n$.
	Existence of $w_\pm$ satisfying, initially, $s_n/n\leq w_-\leq w_+\lesssim (s_n/n)(\log (n/s_n))^{1/2}$ then follows from the intermediate value theorem, since $\tilde{m}$ is continuous, increasing and non-negative and $m_1(\tau,\cdot)$ is continuous and decreasing for each fixed $\tau$ (see \cref{lem:ResultsFromCR18}).
	
To prove the claim, note that asymptotically as $w\to 0$ with $w\geq s_n/n$, by \cref{lem:m1boundsAnyBoundary} we have for some $c,c'>0$
	\begin{alignat*}{3} c(\log (1/w))^{-1/2}&\leq \tilde{m}(w) &&\leq c' (\log (1/w))^{-1/2}, \\
		1/(2w) &\leq m_1(\theta_{0,i},w)&&\leq 1/w.
	\end{alignat*}
	It follows that the left side of \cref{eqn:sn<w-} is of order $n$, while the right side is of the smaller order $n(\log(n/s_n))^{-1/2}$. It also follows that \begin{align*} \sum_{i\in S_0} m_1(\theta_{0,i},C(s_n/n) (\log (n/s_n))^{1/2}) &\leq C^{-1}n (\log (n/s_n))^{-1/2},\\ 
		(1-\nu_n)(n-s_n)\tilde{m}(C(s_n/n)(\log (n/s_n))^{1/2})&\gtrsim (n-s_n) (\log (n/s_n))^{-1/2}\end{align*} for $n$ large, where the suppressed constant does not depend on $C$ (or $\alpha$), 
	so that the right side of \cref{eqn:snlogsn>w+} upper bounds the left for $C$ large enough, as claimed.
	
	To prove the sharper asymptotics, observe by definition that for $w\in \braces{w_-,w_+}$ we have \[ \sum_{i \in S_0} m_1(\theta_{0,i},w)=(1\pm \nu_n)(n-s_n)\tilde{m}(w). \] Since $s_n/n\leq w\leq (s_n/n) (\log (n/s_n))^{1/2}$ we may use the bounds on $m_1$ given above to see that the left side is  $\asymp s_n w^{-1}$. We also note that $\log(1/w)\asymp \log (n/s_n)$, so that the bounds on $\tilde{m}$ given above yield $\tilde{m}(w)\asymp (\log(n/s_n))^{-1/2}$. The result follows, noting also that $s_n/n \to 0$ so $n-s_n \asymp n$. 
\end{proof}

\begin{proof}[Proof of \cref{lem:ExistsLambda+-}]
	We prove the results for $\lambda_+$; the proofs for $\lambda_-$ are almost identical. We begin by showing that any solution to \cref{eqn:def:lambda+} is necessarily unique.
	Indeed, since the right side is positive, any solution necessarily lies in the set
	\[ \braces{\lambda : E_{\theta_0=0} [ \ell_{1,w_+} \mid \ell_{1,w_-}<\lambda]> t}.\]
	On this set, since $\lambda\mapsto E_{\theta_0=0} [ \ell_{1,w_+} \mid \ell_{1,w_-}<\lambda]-t$ is a non-decreasing positive function and $F_{w_-}$ is a strictly increasing non-negative function (see \cref{lem:monotonicity}), the left side of \cref{eqn:def:lambda+} is strictly increasing, yielding the claimed uniqueness of any solution.
	
From here, \cref{lem:AsymptoticsOfFwLambda,lem:ExpectationsSlowlyTo1} yield precise bounds on $F_w(t)$ and $E_{\theta_0=0}[\ell_{1,w} \mid \ell_{1,w'}<t]$ for suitable $w,w',t$ which allow us to conclude the result. Precisely, \cref{lem:ExistsW+-} tells us that $w_-\asymp (s_n/n)(\log (n/s_n))^{1/2}$, so that $\log (1/w_-)\asymp \log (n/s_n)$ and $w_-^{1/2}/\delta_n\to 0$, hence by \cref{lem:AsymptoticsOfFwLambda} (with $c=1/2$) we have for any constant $\kappa>0$ 
\[ F_{w_-}(1-\kappa \delta_n) \asymp \kappa^{-1}\delta_n^{-1} w_- (\log (1/w_-))^{-3/2}\asymp \kappa^{-1} n^{-1} s_n.\] 
[All suppressed constants in this proof will be independent of $\kappa$.]
	Similarly, by \cref{lem:ExpectationsSlowlyTo1} we have \begin{equation}\label{eqn:1-ErhoAsymptotics} 1 - E_{\theta_0=0}[ \ell_{1,w_+} \mid \ell_{1,w_-}<1-\kappa \delta_n] \asymp \kappa  \delta_n\log(1/\delta_n)=\kappa \eps_n.\end{equation} 
Inserting these bounds we see that the left side of \cref{eqn:def:lambda+} is bounded above and below by a constant times
		\[ \kappa ^{-1}(n-s_n)n^{-1} s_n(1-t-O(\kappa \eps_n)) \asymp  \kappa^{-1}  s_n.\]  
	For $\kappa$ large enough (depending on $t$) this is smaller than the right side of \cref{eqn:def:lambda+} and for $\kappa$ small it is larger. The left side is continuous in $\lambda_+$ (see \cref{lem:AsymptoticsOfFwLambda,lem:ExpectationsSlowlyTo1}) while the right side is fixed, so we deduce by the intermediate value theorem the existence of a solution $\lambda_+$ satisfying  $1-C\delta_n\leq \lambda_+\leq 1-c\delta_n$ for constants $C,c>0$, so that \cref{eqn:Lambda+-To1} is proved.
		
	The expectation result \cref{eqn:Elambda+to1} now follows immediately from \cref{eqn:1-ErhoAsymptotics}. The bound \cref{eqn:Flambda+geq} for $F_{w_-}(\lambda_+)$ is obtained by rearranging the definition \cref{eqn:def:lambda+}, inserting the bound for $E_{\theta_0=0}[\ell_{1,w_+} \mid \ell_{1,w_-}\leq \lambda_+]$, and using that $(1-x)^{-1}=1+O(x)$ as $x\to 0$. [For the bound on $F_{w_+}(\lambda_-)$, one also recalls that $\delta_n=o(\eps_n)$.]
	
	Finally, to see that $\lambda_-< \lambda_+$, observe that $\ell_{1,w_-}>\ell_{1,w_+}$ (\cref{lem:monotonicity}) so that for $\lambda>t$, 
	\[\begin{split} & E_{\theta_0=0}[(\ell_{1,w_-}-t)\II\braces{\ell_{1,w_+}<\lambda}]-E_{\theta_0=0}[(\ell_{1,w_+}-t)\II\braces{\ell_{1,w_-}<\lambda}] \\ = & E_{\theta_0=0}[(\ell_{1,w_-}-\ell_{1,w_+})\II\braces{\ell_{1,w_-}<\lambda}] + E_{\theta_0=0}[(\ell_{1,w_-}-t)\II\braces{\ell_{1,w_+}<\lambda\leq \ell_{1,w_-}}]\\ \geq &0.\end{split} \]
	Since \cref{eqn:Lambda+-To1} shows that $\lambda_->t$ for $n$ large, we apply this with $\lambda=\lambda_-$ to deduce that the left side of \cref{eqn:def:lambda+} evaluated at $\lambda_-$ is smaller than its right side:
	\[ \begin{split} &F_{w_-}(\lambda_-)(E_{\theta_0=0}[\ell_{1,w_+} \mid \ell_{1,w_-}<\lambda_-]-t) \\ =& E_{\theta_0=0}[(\ell_{1,w_+}-t) \II\braces{\ell_{1,w_-}<\lambda_-}] \\ \leq & E_{\theta_0=0}[(\ell_{1,w_-}-t)\II\braces{\ell_{1,w_+}<\lambda_-}] \\ & =F_{w_+}(\lambda_-)(E_{\theta_0=0}[\ell_{1,w_-}\mid \ell_{1,w_+}<\lambda_-]-t) \\ & =\frac{ts_n -As_n\max(\nu_n,\rho_n,\delta_n)}{n-s_n}  \\ &< \frac{ts_n+As_n\nu_n}{n-s_n}.\end{split}\]
	Since the right side of \cref{eqn:def:lambda+} is constant and the left side increases with $\lambda_+$ (as noted above when showing uniqueness), this implies that $\lambda_+>\lambda_-$.
\end{proof}

\begin{proof}[Proof of \cref{lem:HatWConcentrates}]
	
	We follow the proof of Lemmas~S-3 in \cite{CR18}, with the essential difference that we do not use the polynomial sparsity but rather the strong signal assumption. Let us prove
\[	P_{\theta_0}( \hat{w}< w_-)  =o(\nu_n)	\] for $\alpha$ large enough in the definition of $\nu_n$, the proof that $P_{\theta_0}( \hat{w}> w_+)  =o(\nu_n)$ being similar.
	
Let $S=L'$ denote the score function as in \cref{eqn:def:Score}. Since $\hat{w}$ maximises $L(w)$, necessarily $S(\hat{w})\leq 0$ or $\hat{w}=1$. If $\hat{w}<w_-$ then only the former may hold, so that applying the strictly monotonic function $S$ (\cref{lem:monotonicity}) we deduce that $\braces{\hat{w}<w_-}=\braces{S(w_-)<S(\hat{w})}\subseteq\braces{S(w_-)<0}$, hence
	\begin{align*}
		P_{\theta_0}( \hat{w}< w_-) \leq  P_{\theta_0}( S(w_-)<0)	&=P_{\theta_0}( S(w_-)-E_{\theta_0}S(w_-) <-E_{\theta_0}S(w_-))\\
		&=P_{\theta_0}\left( \sum_{i=1}^n W_i <-E\right),
	\end{align*}
	where we have introduced the notation 
	 $W_i=\beta(X_i)/(1+w_-\beta(X_i))-m_1(\theta_{0,i},w_-)$ and $E=E_{\theta_0}S(w_-)=\sum_{i=1}^n m_1(\theta_{0,i},w_-)$. For $n$ large $|W_i|\leq \mathcal{M}=2/w_-$ a.s. (see \cref{lem:ResultsFromCR18}), so that we may scale the variables $W_i$ to apply the Bernstein inequality (\cref{lem:Bernstein}) and obtain
	\begin{align*}
		P_{\theta_0}( \hat{w}< w_-) \leq e^{-0.5 E^2/(V_2 + \mathcal{M}E/3)},
	\end{align*}
	where $V_2=\sum_{i=1}^n \Var(W_i)\leq \sum_{i=1}^n m_2(\theta_{0,i},w_-)$, for $m_2(\theta_{0,i},w)=E_{\theta_0} (\beta(X_i)/[1+w\beta(X_i)])^2 $. 
	In view of the definition \eqref{eqn:def:w-} of $w_-$, we have
	\[
	E= \sum_{i\in S_0} m_1(\theta_{0,i},w_-)- (n-s_n)\tilde{m}(w_-) = \nu_n(n-s_n) \tilde{m}(w_-).
	\]
	We also note, using the strong signal assumption and the bounds on $m_2$ in \cref{lem:ResultsFromCR18} that for some constants $C,M_0>0$ and $n$ larger than some universal threshold,
	\begin{align*}
		V_2
		&\leq \sum_{i:|\theta_{0,i}|>M_0}  m_2(\theta_{0,i},w_-)+\sum_{i:\theta_{0,i}=0}  m_2(0,w_-)\\
		&\leq \frac{C}{w_-}\sum_{i\in S_0} m_1(\theta_{0,i},w_-) + C(n-s_n)\frac{\bar{\Phi}(\zeta(w_-))}{w_-^2},
	\end{align*}
with $\zeta$ defined as in \cref{eqn:def:zeta}. By a standard normal tail bound and the definition of $\zeta$, we have $\bar{\Phi}(\zeta(w_-))\asymp \phi(\zeta(w_-))/\zeta(w_-) \asymp w_- g(\zeta(w_-)) /\zeta(w_-)$, which is of order $w_- \tilde{m}(w_-)/ \zeta(w_-)^2$ because $\tilde{m}(w_-)\asymp \zeta(w_-) g(\zeta(w_-))$ (see \cref{lem:ResultsFromCR18}). 
Using the latter, and the fact that $\zeta(w_-)\to \infty$ (\cref{lem:ResultsFromCR18}), in combination with \eqref{eqn:def:w-} gives
	\[
	V_2\lesssim  n w_-^{-1}  \tilde{m}(w_-) + n w_-^{-1} \tilde{m}(w_-)/ \zeta(w_-)^2 \lesssim  n w_-^{-1}  \tilde{m}(w_-),
	\]
 so that 
	\[
	\frac{V_2+ \mathcal{M}E/3}{E^2}\lesssim  \frac{n w_-^{-1}  \tilde{m}(w_-)}{ (\nu_n(n-s_n) \tilde{m}(w_-))^2} + \frac{1}{ w_- \nu_n(n-s_n) \tilde{m}(w_-)}\lesssim  \frac{1}{ \nu^2_n n w_-\tilde{m}(w_-)} .
	\]
		This implies that
	$
	P_{\theta_0}( \hat{w}< w_-) \leq e^{-c \nu^2_n n w_-\tilde{m}(w_-)}
	$
	for some constant $c>0$.  
	Now, by Lemma~\ref{lem:ExistsW+-}, we have $n w_-\tilde{m}(w_-)\asymp s_n$. Hence, recalling the definition $\nu_n=\alpha s_n^{-1/2}(\log s_n)^{1/2}$ from eq.~\eqref{eqn:def:nu}, we deduce that $\nu^2_n n w_-\tilde{m}(w_-)\geq \log s_n$ if the constant $\alpha$ is large enough, and hence the above probability is bounded above by $s_n^{-1/2}=o(\nu_n)$. 
\end{proof}

\begin{proof}[Proof of \cref{lem:HatLambdaConcentrates}]
	Let $\Bb$ be an event on which, with $K_n:=\#\braces{i\in S_0 : \ell_{i,w_-}>\delta_n}$,
	\begin{equation}\label{eqn:eventB}\begin{split} \hat{w}&\in(w_-,w_+),\\ 
			K_n&\leq s_n (\rho_n+\nu_n), \\
			V_{\lambda_+,w_-}&\leq E[V_{\lambda_+,w_-}]+as_n\nu_n,  \\
			V_{\lambda_-,w_+}&\geq E[V_{\lambda_-,w_+}]-a
			s_n\nu_n,
		\end{split}
	\end{equation}
and whose complement has probability $P_{\theta_0}(\Bb^c)=o(\nu_n)$;
 note that such an event exists by \cref{lem:HatWConcentrates,lem:lvalsSmallForTrueSignals,lem:VLambdaWConcentrates}, the proofs of which are independent of \cref{lem:HatLambdaConcentrates}. 
	Recall that $\hat{\lambda}$ is characterised by the posterior FDR: \[\postFDR_{\hat{w}}(\vphi_{\lambda,\hat{w}}):=\frac{\sum_{i=1}^n \ell_{i,\hat{w}} \II\braces{(\vphi_{\lambda,\hat{w}})_i=1}
	}{1\vee (\sum_{i=1}^n \II\braces{(\vphi_{\lambda,\hat{w}})_i=1})
	} \leq t \iff \lambda\leq \hat{\lambda}.\] Thus, it is enough to bound the posterior  FDRs of $\vphi_{\lambda_-,\hat{w}},\vphi_{\lambda_+,\hat{w}}$ above and below respectively by $t$. 	We prove the upper and lower bound separately, which suffices by a union bound. 

	{\em Upper bound, $\postFDR_{\hat{w}}(\vphi_{\lambda_+,\hat{w}})>t$ with probability at least $1-o(\nu_n)$.}
	On the event $\Bb$, monotonicity (see \cref{lem:monotonicity}) allows us to deduce that
	\begin{equation}\label{eqn:postFDRlambda+lowerbound} \begin{split} \postFDR_{\hat{w}}(\vphi_{\lambda_+,\hat{w}}) &\geq \postFDR_{\hat{w}} (\vphi_{\lambda_+,w_-})  \\
			&\geq \postFDR_{w_+}(\vphi_{\lambda_+,w_-}) \\
			&\geq \frac{\sum_{i\not \in S_0} \ell_{i,w_+}\II\braces{\ell_{i,w_-}<\lambda_+}}{s_n+V_{\lambda_+,w_-}},
		\end{split} 
	\end{equation} where to obtain the last line we have used that $\sum_{i\in S_0}\II\braces{\ell_{i,w_-}<\lambda_+}\leq s_n$ and $\sum_{i\in S_0}\ell_{i,w_+}\II\braces{\ell_{i,w_-}<\lambda_+}\geq 0$.
	We apply Bernstein's inequality (see \cref{lem:Bernstein}) with, for some $a$ (indeed, the same $a$ as in \cref{eqn:eventB}, coming originally from \cref{lem:VLambdaWConcentrates}, works),
	\[ u= a s_n\nu_n, \qquad U_i=-\ell_{i,w_+}\II\braces{\ell_{i,w_-}<\lambda_+},~i\not\in S_0.\] 
 Note that 
	\[\begin{split} \sum_{i\not \in S_0} E[U_i]&=-E[V_{\lambda_+,w_-}] E_{\theta_0=0}[\ell_{1,w_+} \mid \ell_{1,w_-} < \lambda_+], \\ \sum_{i\not \in S_0} \Var(U_i)&\leq E V_{\lambda_+,w_-} \asymp 
		s_n. \end{split} \] 
	For $a$ large enough we deduce that
	\begin{multline}\label{eqn:BernsteinForUpperBoundNumerator}P_{\theta_0}\brackets[\Big]{\sum_{i\not \in S_0} \ell_{i,w_+} \II\braces{\ell_{i,w_-}<\lambda_+}< \\ E[V_{\lambda_+,w_-}] E_{\theta_0=0} [\ell_{1,w_+} \mid \ell_{1,w_-}<\lambda_+]-a
			s_n\nu_n} \leq  s_n^{-1/2}.\end{multline} 
	Then by a union bound we see that on an event $\Cc\subset \Bb$ of probability at least $P(\Bb)-s_n^{-1/2}=1-o(\nu_n)$, the numerator in the final line of \cref{eqn:postFDRlambda+lowerbound} is lower bounded by \begin{equation*}\label{eqn:NumeratorBound}\begin{split}
			&E[V_{\lambda_+,w_-}] E_{\theta_0=0}[\ell_{1,w_+} \mid \ell_{1,w_-} <\lambda_+] -a
			s_n\nu_n  \\ =&(n-s_n)F_{w_-}(\lambda_+)E_{\theta_0=0}[\ell_{1,w_+} \mid \ell_{1,w_-}<\lambda_+]-a
			s_n\nu_n.\end{split} 
		\end{equation*}
	  Recalling also that on $\Bb$ we have \[ V_{\lambda_+,w_-}\leq E[V_{\lambda_+,w_-}] + 
	  a s_n\nu_n=(n-s_n)F_{w_-}(\lambda_+)+
	  a s_n\nu_n,\] we deduce that
	\[\postFDR_{\hat{w}}(\vphi_{\lambda_+,\hat{w}}) \geq \II_\Cc \frac{ (n-s_n)F_{w_-}(\lambda_+) E_{\theta_0=0}[\ell_{1,w_+} \mid \ell_{1,w_-}<\lambda_+]-
		as_n\nu_n}{s_n+(n-s_n)F_{w_-}(\lambda_+) +a 
		s_n\nu_n}.\] Substituting for the first term in the numerator from the definition \cref{eqn:def:lambda+}, we find that, for $A>(1+t)a$,
	\[\postFDR_{\hat{w}}(\vphi_{\lambda_+,\hat{w}}) \geq \II_\Cc \brackets[\Big]{t + \frac{
			(A-(1+t)a)s_n\nu_n}{s_n+(n-s_n)F_{w_-}(\lambda_+)+a s_n \nu_n}} >t\II_\Cc,\] so that indeed $\hat{\lambda}\leq \lambda_+$, at least for $n$ large enough, on the event $\Cc$.
	
	{\em Lower bound, $\postFDR_{\hat{w}}(\vphi_{\lambda_-,\hat{w}})\leq t$ with probability at least $1-o(\nu_n)$.} 
	On the event $\Bb$, recalling \cref{eqn:eventB}, and using monotonicity of the $\ell$-values (\cref{lem:monotonicity}) and the fact that $\lambda_-$ is bounded away from zero (\cref{lem:ExistsLambda+-}), we see that \[ \begin{split} \#\braces{i\in S_0 :(\vphi_{\lambda_-,w_+})_i=0} &\leq \#\braces{i\in S_0 : \ell_{i,w_+}> \delta_n} \\ &\leq \#\braces{i\in S_0 : \ell_{i,w_-}> \delta_n}=K_n\leq s_n(\rho_n+\nu_n).\end{split}\] Since $\ell_{i,w_+}\leq 1$ for all $i$, we also note that
	\[ \sum_{i\in S_0} \ell_{i,w_+} \leq K_n+ \sum_{i\in S_0} \delta_n \leq s_n(\rho_n+\nu_n+\delta_n).\]   Then on $\Bb$, monotonicity arguments as used for the upper bound yield
	\begin{equation} \label{eqn:postFDRlambda-upperbound}
		\begin{split}
			\postFDR_{\hat{w}}(\vphi_{\lambda_-,\hat{w}}) 
			& \leq \postFDR_{w_-}(\vphi_{\lambda_-,w_+}) \\ 
			&\leq \frac{\sum_{i\not\in S_0} \ell_{i,w_-}\II\braces{\ell_{i,w_+}<\lambda_-}+s_n(\rho_n+\nu_n+\delta_n)}{s_n-s_n(\rho_n+\nu_n) +V_{\lambda_-,w_+}}.
		\end{split}
	\end{equation}
	
	Applying Bernstein's inequality as for the upper bound, here with variables $U_i=\ell_{i,w_-}\II\braces{\ell_{i,w_+}<\lambda_-},~i\not\in S_0$, we deduce that there is an event $\Cc'\subset \Bb$ of probability at least $1-o(\nu_n)$ such that
	\[\begin{split} &\postFDR_{\hat{w}}(\vphi_{\lambda_-,\hat{w}}) \II_{\Cc'} \\ \leq &  \frac{(n-s_n)F_{w_+}(\lambda_-) E_{\theta_0=0}[\ell_{1,w_-}\mid \ell_{1,w_+} < \lambda_-] + s_n(\rho_n+(1+a
		)\nu_n+\delta_n)}
	{s_n+(n-s_n)F_{w_+}(\lambda_-)-s_n(\rho_n+(1+a)
		\nu_n)}.\end{split}\] 
	Substituting for $(n-s_n)F_{w_+}(\lambda_-)E_{\theta_0=0}[\ell_{1,w_-}\mid\ell_{1,w_+}<\lambda_-]$ in the numerator from the definition \cref{eqn:def:lambda-} of $\lambda_-$, the right side is upper bounded by $t$ if $A$ is large enough, so that indeed $\lambda_-\leq\hat{\lambda}$ on $\Cc'$. 
\end{proof}

\begin{proof}[Proof of \cref{lem:lvalsSmallForTrueSignals}]
	The idea is that for $w$ of order $(s_n/n)(\log(n/s_n))^{1/2}$ (which is the case for $w=w_-,w_+$ or, with high probability, $w=\hat{w}$), the function $\ell(x;w)$ defining the $\ell$-values is vanishingly small when $x>\sqrt{2\log(n/s_n)+u_n}$ for $u_n$ tending to infinity slowly. In contrast, for $i\in S_0$ we have $\abs{X_i}>\sqrt{2\log(n/s_n)}+v_n/2$ with probability tending to 1 since $X_i-\theta_{0,i} = \eps_i\sim \Nn(0,1)$ is bounded in probability. Taylor expanding the square root reveals $\sqrt{2\log(n/s_n)+u_n}=\sqrt{2\log(n/s_n)}+o(1)$ and thus will yield the claim.
		
	Let $u_n=5\log \log (n/s_n)$ and define 
	\[ S_2 = \braces{i\in S_0 : \abs{X_i}>\sqrt{2 \log (n/s_n)+u_n}}.\]
	First, we show that $S_2\subset S_1$. 
	From \cref{lem:ResultsFromCR18} we have, for $\xi=(\phi/g)^{-1}$,
	\begin{equation*} 
		\xi(u) \leq \brackets[\Big]{ 2\log(1/u) + 2\log\log (1/u) +6\log 2}^{1/2}.\end{equation*}
	The right side is decreasing in $u$ so, recalling that $w_-\geq s_n/n$ (\cref{lem:ExistsW+-}), we see that $\xi$ evaluated at $u=(w_-/(2\log(n/s_n)))$ is upper bounded by the right side evaluated at $u=s_n/(2 n\log (n/s_n))$, hence 
	\begin{align*} &\xi\brackets[\Big]{\frac{w_-}{2\log(n/s_n)}} \\ 
		\leq& \sqrt{ 2\log(n/s_n) + 4\log \log (n/s_n) +2\log\log\log(n/s_n)+8\log2+2\log\log 2} \\
		\leq& \sqrt{2\log(n/s_n)+u_n}, 
	\end{align*}
	for $n$ large.
	Consequently we see that if $\abs{x}>\sqrt{2\log(n/s_n)+u_n}$, then \[\phi(x)/g(x)= \xi^{-1}(x) > w_-/(2\log(n/s_n))=\tfrac{1}{2}w_-\delta_n,\] so that
	\begin{equation*}
		\ell(x;w_-)
		=\brackets[\Big]{1+\frac{w_-}{1+w_-} \frac{g}{\phi}(x)}^{-1} \leq \delta_n,
	\end{equation*} and indeed $S_2\subset S_1$.
	
		Next, observe, by Taylor expanding, that $\sqrt{2\log(n/s_n)+u_n}=\sqrt{2\log(n/s_n)}+o(1)$. We deduce that for $i\in S_0\setminus S_2$, necessarily the noise variable $\eps_i$ in \cref{eqn:def:GaussianSequenceModel} satisfies $\abs{\eps_i}>v_n/2,$
	so that $\abs{S_0\setminus S_1}\leq \abs{S_0\setminus S_2}\leq N,$ where $N$ is the binomial $N=\#\braces{i \in S_0 : \abs{\eps_i}>v_n/2}$. Applying Bernstein's inequality (\cref{lem:Bernstein}) with \[U_i=\II\braces[\big]{ \abs{\eps_i}>v_n/2},\quad u=\max\brackets[\big]{EN,\nu_n}\geq \sum_{i\in S_0} \Var U_i,\] we see that
	\[\Pr(N>EN +u)\leq \exp\brackets[\Big]{-\frac{u^2/2}{u+u/3}}=o(\nu_n),\] for large enough constant $\alpha$ in the definition \cref{eqn:def:nu} of $\nu_n$.
	Finally, note that $EN=2s_n\bar{\Phi}(v_n/2)\leq s_n \rho_n$ for $\rho_n=e^{-v_n^2/9}$ as defined in \cref{eqn:def:rho}, at least for $n$ large, as a consequence of the standard tail bound $\bar{\Phi}(x)\asymp \phi(x)/x\ll e^{-x^2/2}$ as $x\to \infty$.
\end{proof}

\begin{proof}[Proof of \cref{lem:VLambdaWConcentrates}]
	\Cref{lem:ExistsW+-} tells us that $w\asymp (s_n/n)(\log (n/s_n))^{1/2}$ for $w\in \braces{w_-,w_+}$, and \cref{lem:ExistsLambda+-} tells us that $1-\lambda\asymp \delta_n$ for $\lambda\in \braces{\lambda_-,\lambda_+}$. Then $V_{\lambda,w}=\sum_{i\not \in S_0} \II\braces{\ell_{i,w}<\lambda}$ follows a binomial distribution, whose mean we deduce by \cref{lem:AsymptoticsOfFwLambda} satisfies
	\[ \begin{split} E[V_{\lambda,w}] = (n-s_n)F_w(\lambda) &\asymp (n-s_n)w(1-\lambda)^{-1} (\log(1/w))^{-3/2} \\ &\asymp 
	(n-s_n)w(\log(n/s_n))^{-1/2}, \end{split} \] 
	so that again appealing to \cref{lem:ExistsW+-}, we have \( E[V_{\lambda,w}]\asymp 
	s_n.\)
	We apply Bernstein's inequality \cref{lem:Bernstein} with, for some $a=a(t)$,
	\[ U_i= \II\braces{\ell_{i,w}<\lambda}, \qquad u= a s_n\nu_n.\]
	Then $\sum_{i\not\in S_0} \Var(U_i) \leq E[V_{\lambda,w}]\asymp s_n$ so that for a constant $C$, larger than $1/2$ for $a$ large enough,
	\[ P_{\theta_0}(\abs{V_{\lambda,w}-E[V_{\lambda,w}]}\geq u) \leq 2\exp\brackets{-C\log s_n}\leq 2s_n^{-1/2}=o(\nu_n). \qedhere\]
\end{proof}

\begin{proof}[Proof of \cref{lem:EVw+closetoEVw-}] 	We prove the control \cref{eqn:EV++} for $E[V_{\lambda_+,w_+}]$; the proof for $E[V_{\lambda_-,w_-}]$ is almost identical. The starting point is the bound $w_+/w_-=O(1)$ from \cref{lem:ExistsW+-} and we combine this with control of the functions making up the expectation. By \cref{lem:AsymptoticsOfFwLambda} we have $E V_{\lambda,w}= (n-s_n) F_w(\lambda)=2(n-s_n)\bar{\Phi}(\xi(r(w,\lambda)))$ for any $\lambda,w\in(0,1)$, where we recall the definitions $r(w,\lambda)=w\lambda(1-w)^{-1}(1-\lambda)^{-1}$, $\xi=(\phi/g)^{-1}$, so that our goal is to bound \[\frac{E[V_{\lambda_+,w_+}]}{E[V_{\lambda_+,w_-}]}-1=\frac{\bar{\Phi}\brackets[\big]{\xi(r(w_+,\lambda_+))}}{\bar{\Phi}\brackets[\big]{\xi(r(w_-,\lambda_+))}}-1.\]

Write $r_\pm=r(w_\pm,\lambda_+)$ and $\xi_\pm=\xi(r_\pm)$ (the notation $\xi_+$ is to link to $r_+$, not to claim that $\xi_+\geq \xi_-$). As a consequence of \cref{lem:ExistsW+-,lem:ExistsLambda+-}, 
	\[\log (1/r_-)\asymp \log (1/r_+) \asymp \log (n/s_n)=\delta_n^{-1}.\]
	Recalling that $\xi(u)\sim (-2\log u)^{1/2}$ as $u\to 0$ (see \cref{lem:ResultsFromCR18}) it follows that $\xi_\pm\to \infty$, hence by a standard normal tail bound (also in \cref{lem:ResultsFromCR18}) we have 
	\[ 0\leq \frac{\bar{\Phi}(\xi_+)}{\bar{\Phi}(\xi_-)}-1 \leq \frac{(1+\xi_-^2)}{\xi_-^2}\frac{\xi_-}{\xi_+}\frac{\phi(\xi_+)}{\phi(\xi_-)}-1= O\brackets[\Big]{\max \brackets[\Big]{\frac{1}{\xi_-^2},\frac{\xi_-}{\xi_+}-1,\frac{\phi(\xi_+)}{\phi(\xi_-)}-1}},\] provided the right hand side tends to zero, using that for $a_n,b_n\to 0$, $(1+a_n)(1+b_n)-1=O(\max(a_n,b_n))$. That $\xi(u)\sim (-2\log u)^{1/2}$ as $u\to 0$ implies $\xi_-^{-2}-1 =O( (\log (1/r_-))^{-1})=O(\delta_n)$.
Next, by \cref{lem:xi+-differencebounded} we have
		\[\xi_-^2-\xi_+^2 = O(1),\]
hence
	\[ \frac{\xi_-}{\xi_+} -1 = \frac{\xi_-^2-\xi_+^2}{\xi_+\xi_-+\xi_+^2}=O\brackets[\Big]{(\log \tfrac{n}{s_n})^{-3/2}}=o(\delta_n).\] 
	It remains to control $\phi(\xi_+)/\phi(\xi_-)-1$. 
	By the definition of $\xi$ we have
	\[ \frac{\phi(\xi_+)}{\phi(\xi_-)} = \frac{r_+g(\xi_+)}{r_- g(\xi_-)}.\] 
	\Cref{lem:w+/w-to1} tells us that \[ 
	\frac{r_+}{r_-}-1 = O(\max(\nu_n,\rho_n)),\] so that it suffices to show $g(\xi_+)/g(\xi_-)-1=O(\delta_n)$. From the explicit definition \cref{eqn:def:gInQuasiCauchy} of $g$, we have
	\[ \frac{g(\xi_+)}{g(\xi_-)}-1= \frac{\xi_-^2}{\xi_+^2} \frac{1-e^{-\xi_+^2/2}}{1-e^{-\xi_-^2/2}}-1.\] Observe that, for $n$ large, \[ \frac{1-e^{-\xi_+^2/2}}{1-e^{-\xi_-^2/2}}-1 = \frac{e^{-\xi_-^2/2}-e^{-\xi_+^2/2}}{1-e^{-\xi_-^2/2}}\leq 2e^{-\xi_-^2/2}.\]
	The lower bound on $\xi$ in \cref{lem:ResultsFromCR18} implies that $\xi(u)\geq \sqrt{2\log(1/u)}$ for $u$ small, so that $e^{-\xi_-^2}\leq r_-^2,$ which is of smaller order than $\delta_n$ (note that $r_-\asymp (s_n/n)(\log (n/s_n))^{3/2}$ as a consequence of \cref{lem:ExistsW+-,lem:ExistsLambda+-}). Noting that the bound attained above for $\xi_-/\xi_+-1$ also bounds $\xi_-^2/\xi_+^2-1$, we deduce that $\phi(\xi_+)/\phi(\xi_-)-1$ is suitably bounded and the \namecref{lem:EVw+closetoEVw-} follows. 
\end{proof}

\section{Sparsity preserving procedures and optimality of the boundary} \label{sec:sparsity-preserving}

Here we make precise the claim of \cref{sec:optimality} that the condition $v_n\to\infty$ in \cref{thm:FDRConvergenceNoRate} cannot be relaxed at all without weakening the conclusion of the \namecref{thm:FDRConvergenceNoRate} (and correspondingly for \cref{thm:QvalueControlsFDR}).

Define a multiple testing procedure\footnote{Strictly, one must consider \emph{sequences} of testing procedures $\vphi^{(n)}$ and sets $\Theta_n$ in the definition; we leave this implicit.} $\vphi$ to be \emph{sparsity preserving at a level $A_n$ over a set $\Theta$} if 
\begin{equation}\label{eqn:def:sparsity-preserving} \sup_{\theta_0\in\Theta} P_{\theta_0}\big(\sum_{i\leq n} \vphi_i >A_ns_n\big)\to 0.\end{equation}
One has the following (recall the definition \cref{eqn:def:StrongSignalClass} of $\ell_0(s_n;b)$).
\begin{theorem*}[Theorem 3 in \cite{ACR21lb}]
	For fixed $b\in\RR$ and a sequence $(A_n)_{n\in\NN}$ with $A_n\in [2,e^{(\log(n/s_n))^{1/4}}]$, for any $\vphi$ satisfying \cref{eqn:def:sparsity-preserving} with $\Theta=\ell_0(s_n;b)$, we have
	\[ \sup_{\theta_0 \in \ell_0(s_n;b)} \FNR(\vphi;\theta_0) = \bar{\Phi}(b)+o(1).\]
	In particular, the FNR of such a procedure is bounded away from zero for $n$ large.
\end{theorem*}
		
We now prove that both the $\Cl$- and $q$-value procedures are sparsity preserving, so that the above theorem applies to show the conclusions of \cref{thm:FDRConvergenceNoRate,thm:QvalueControlsFDR} are impossible for $v_n=b\in\RR$ fixed.
The following lemma will be helpful.
\begin{lemma}\label{lem:sparsity-preserving}
	Any procedure whose number of false positives $V$ satisfies, for some event $\Aa_n$,
\[ \sup_{\theta_0\in \Theta} E[V\II_{\Aa_n}] = O(s_n), \quad \sup_{\theta_0\in \Theta} P_{\theta_0}(\Aa_n^c)=o(1),\] is sparsity preserving at level $A_n$ on the set $\Theta$ for any sequence $A_n\to\infty$.
\end{lemma}
\begin{proof}
	Simply note that $\sum_{i\leq n}\vphi_i \leq s_n + V$, so that an application of Markov's inequality yields
	\[ P_{\theta_0}\big(\sum_{i=1}^{n}\vphi_i >A_n s_n\big)\leq P_{\theta_0}(\Aa_n^c)+ (A_n-1)^{-1} s_n^{-1}E_{\theta_0} [V\II_{\Aa_n}] \to 0.\qedhere\]
\end{proof}

We will also use the following result which extends key conclusions of \cref{lem:ExistsW+-,lem:HatWConcentrates} to the weaker signal class $\ell_0(s_n;b)$.

\begin{lemma}[Lemmas 9 and 10 in \cite{ACR21lb}]\label{lem:ExistsW+-WeakSignal}
	There exist $w_-\leq w_+$ for which
	\[ \sup_{\theta_0\in \ell_0(s_n;b)} P_{\theta_0}(\hat{w}\not \in (w_-,w_+))\to 0,\] such that $w_\pm \asymp (s_n/n)(\log(n/s_n))^{1/2}$ and, for some constant $\nu \in (0,1/2)$,
	\[ \sum_{i\in S_0} m_1(\theta_{0,i},w_+) = (1-\nu)(n-s_n)\tilde{m}(w_+).\] 
\end{lemma}

Armed with the above results, we prove that the procedures considered herein are sparsity preserving for any sequence $A_n\to\infty$ over the sets $\ell_0(s_n;b)$ for any $b\in \RR$ fixed.
\paragraph{$q$-value procedure}

With notation as in \cref{sec:proofs:qval}, by \cref{eqn:EV'} we have for $w$ small enough 
\begin{equation*} \begin{split}
	E_{\theta_0} V'_{w}=2(n-s_n)\overline{\Phi}(\chi(r(w,t))) &= (n-s_n)r(w,t)2\overline{G}(\chi(r(w,t))) \\&\leq C(t) (n-s_n) w \tilde{m}(w). 
\end{split}
\end{equation*}
Noting that $m_1(\theta_{0,i},w_+)\leq 1/w_+$ (see the proof of \cref{lem:m1boundsAnyBoundary} and observe that the relevant part holds for $v_n=b\in \RR$ fixed), we deduce that for $w=w_+$ defined as in \cref{lem:ExistsW+-WeakSignal} we have
\[ E_{\theta_0} V'_{w_+} \leq  C(t,\nu) \sum_{i\in S_0}w_+m_1(\theta_{0,i},w_+)\leq C'(t,\nu)s_n,\] so that 
\cref{lem:sparsity-preserving} with $\mathcal{A}_n=\{\hat w\leq w_+\}$ and monotonicity of $E_{\theta_0} V'_{w}$ in $w$ yield the claimed sparsity-preservingness of $\vphi^{\qval}$.

\paragraph{$\Cl$-value procedure}
With \cref{lem:ExistsW+-WeakSignal} replacing \cref{lem:ExistsW+-,lem:HatWConcentrates}, the proofs of \cref{lem:ExistsLambda+-,lem:HatLambdaConcentrates,lem:EVw+closetoEVw-} go through unchanged to show that there exists $\lambda_+$ such that
$(n-s_n)F_{w_+}(\lambda_+)\leq Cs_n$ for a constant $C>0$ and $P_{\theta_0}(\hat{\lambda}>\lambda_+)=o(1)$. [Note the proofs for $\lambda_-$, which would require an adapted version of \cref{lem:lvalsSmallForTrueSignals}, are not needed here.] Then \cref{lem:sparsity-preserving} yields the claimed sparsity-preservingness of $\vphi^{\Cl}$, since 
\[E_{\theta_0}[ V_{\hat{\lambda},\hat{w}}\II\braces{\hat{\lambda}\leq \lambda_+, \hat{w}<w_+}]\leq E_{\theta_0}V_{\lambda_+,w_+} = (n-s_n)F_{w_+}(\lambda_+),\] 
which is $O(s_n)$ as noted just above.

\begin{remark}
	Similar arguments reveal that $\sup_{\theta_0 \in \ell_0(s_n,b)} \FDR(\vphi,\theta_0)\leq ct+o(1)$ for some constant $c>0$ for $b\in \RR$ fixed. 
	Indeed, in the proof of Theorem 5(i) in \cite{ACR21lb}, noting that $\Lambda_\infty$ defined therein equals $\bar{\Phi}(b)$ in the current setting, it is argued that on an event of probability tending to 1
	\[ \#\braces{i\in S_0: \ell_{i,\hat{w}}<t}\geq s_n(1-\bar{\Phi}(b))/2=s_n\Phi(b)/2.\]
	Recalling that $\hat{\lambda}\geq t$ and using monotonicity as in the proof of \cref{thm:FDRConvergenceNoRate}, one has
	\[ \FDP(\vphi^{\Cl};\theta_0) \leq \frac{ V_{\lambda_+,w_+}}{V_{\lambda_+,w_+}+ s_n\Phi(b)/2} +o_p(1).\]
	Taking expectations and using Jensen's inequality for the concave map $x\mapsto x/(s_n\Phi(b)/2+x)$ we obtain
	\[ \FDR(\vphi^{\Cl};\theta_0) \leq \frac{E V_{\lambda_+,w_+}}{E V_{\lambda_+,w_+}+s_n\Phi(b)/2}+o(1).\] Inserting the bound $E_{\theta_0}V_{\lambda_+,w_+}\leq Cs_n$ yields the claim. (We make no attempt to obtain the sharp constant $c$, which we believe will be strictly larger than 1 in this setting.)
\end{remark}
\section*{Acknowledgements}
This work has been supported by ANR-16-CE40-0019 (SansSouci), ANR-17-CE40-0001 (BASICS) and by the GDR ISIS through the "projets exploratoires" program (project TASTY). It was mostly completed while K.A.\ was at Université Paris-Saclay, supported by a public grant as part of the Investissement d'avenir project, reference ANR-11-LABX-0056-LMH, LabEx LMH.
We thank an associate editor and two anonymous referees for their insightful comments, which helped us improve the paper.
%
%

\begin{thebibliography}{00}
	\bibitem{AGC20}
	K.~Abraham, I.~Castillo, and E.~Gassiat. Multiple testing in
	nonparametric {H}idden {M}arkov models: {A}n {E}mpirical {B}ayes approach.
	2021. Arxiv preprint \arxivurl{2101.03838}.
	
	\bibitem{ACR21lb}
	K.~Abraham, I.~Castillo, and E.~Roquain. Sharp multiple testing
	boundary for sparse sequences. 2021. Arxiv preprint
	\arxivurl{2109.13601}.
	
	\bibitem{amar2017extracting}
	D.~Amar, R.~Shamir, and D.~Yekutieli. Extracting replicable associations
	across multiple studies: {E}mpirical {B}ayes algorithms for controlling
	the false discovery rate. {\em PLoS computational biology},
	13(8):e1005700, 2017.
	
	\bibitem{AC2017}
	E.~Arias-Castro and S.~Chen. Distribution-free multiple testing.
	{\em Electron. J. Stat.}, 11(1):1983--2001, 2017.
	\MR{3651021}
	
	\bibitem{AS2015}
	D.~Azriel and A.~Schwartzman. The empirical distribution of a
	large number of correlated normal variables. {\em Journal
		of the American Statistical Association}, 110(511):1217--1228, 2015.
	\MR{3420696}
	
	\bibitem{bcg20}
	S.~Banerjee, I.~Castillo, and S.~Ghosal. Bayesian inference in
	high-dimensional models. 2021. Book chapter to appear
	in Springer volume on data science, Arxiv preprint \arxivurl{2101.04491}.
	
	\bibitem{BC2015}
	R.~F. Barber and E.~J. Cand\`{e}s. Controlling the false discovery
	rate via knockoffs. {\em Ann. Statist.}, 43(5):2055--2085,
	2015.
	\MR{3375876}
	
	\bibitem{barber2019knockoff}
	R.~F. Barber, E.~J. Cand{\`{e}}s, et~al. A knockoff filter for
	high-dimensional selective inference. {\em The Annals of
		Statistics}, 47(5):2504--2537, 2019.
	\MR{3988764}
	
	\bibitem{belitser21}
	E.~Belitser and N.~Nurushev. Uncertainty quantification for robust
	variable selection and multiple testing. 2021. Arxiv preprint
	\arxivurl{2109.09239}.
	
	\bibitem{BH1995}
	Y.~Benjamini and Y.~Hochberg. Controlling the false discovery
	rate: {A} practical and powerful approach to multiple testing.
	{\em J. Roy. Statist. Soc. Ser. B}, 57(1):289--300, 1995.
	\MR{1325392}
	
	\bibitem{BKY2006}
	Y.~Benjamini, A.~M. Krieger, and D.~Yekutieli. Adaptive linear
	step-up procedures that control the false discovery rate. {\em Biometrika},
	93(3):491--507, 2006.
	\MR{2261438}
	
	\bibitem{BY2001}
	Y.~Benjamini and D.~Yekutieli. The control of the false discovery
	rate in multiple testing under dependency. {\em Ann. Statist.},
	29(4):1165--1188, 2001.
	\MR{1869245}
	
	\bibitem{BR2009}
	G.~Blanchard and E.~Roquain. Adaptive false discovery rate control
	under independence and dependence. {\em J. Mach. Learn. Res.},
	10:2837--2871, 2009.
	\MR{2579914}
	
	\bibitem{bogdanslope15}
	M.~Bogdan, E.~van~den Berg, C.~Sabatti, W.~Su, and E.~J. Cand\`{e}s.
	S{LOPE}---adaptive variable selection via convex optimization.
	{\em Ann. Appl. Stat.}, 9(3):1103--1140, 2015.
	\MR{3418717}
	
	\bibitem{SC2019}
	T.~T. Cai, W.~Sun, and W.~Wang. Covariate-assisted ranking and
	screening for large-scale two-sample inference. {\em J. R.
		Stat. Soc. Ser. B. Stat. Methodol.}, 81(2):187--234, 2019.
	\MR{3928141}
	
	\bibitem{cm18}
	I.~Castillo and R.~Mismer. Empirical {B}ayes analysis of spike
	and slab posterior distributions. {\em Electron. J. Stat.},
	12(2):3953--4001, 2018.
	\MR{3885271}
	
	\bibitem{CR18}
	I.~Castillo and E.~Roquain. On spike and slab empirical {B}ayes
	multiple testing. {\em Ann. Statist.}, 48(5):2548--2574,
	2020.
	\MR{4152112}
	
	\bibitem{cs20}
	I.~Castillo and B.~Szab\'{o}. Spike and slab empirical {B}ayes
	sparse credible sets. {\em Bernoulli}, 26(1):127--158, 2020.
	\MR{4036030}
	
	\bibitem{CDH2018}
	X.~Chen, R.~W. Doerge, and J.~F. Heyse. Multiple testing with
	discrete data: {P}roportion of true null hypotheses and two adaptive {FDR}
	procedures. {\em Biom. J.}, 60(4):761--779, 2018.
	\MR{3830957}
	
	\bibitem{encode2007}
	E.~P. Consortium et~al. Identification and analysis of functional
	elements in 1\% of the human genome by the encode pilot project.
	{\em Nature}, 447(7146):799, 2007.
	
	\bibitem{Dic2014}
	T.~Dickhaus. {\em Simultaneous statistical inference: With
		applications in the life sciences}. Springer, Heidelberg, 2014.
	\MR{3184277}
	
	\bibitem{DDR2018}
	S.~D{\"{o}}hler, G.~Durand, and E.~Roquain. New {FDR} bounds for
	discrete and heterogeneous tests. {\em Electron. J. Statist.},
	12(1):1867--1900, 2018.
	\MR{3813600}
	
	\bibitem{Dur2017}
	G.~Durand. Adaptive {$p$}-value weighting with power optimality.
	{\em Electron. J. Stat.}, 13(2):3336--3385, 2019.
	\MR{4010982}
	
	\bibitem{Efron2004}
	B.~{Efron}. {Large-scale simultaneous hypothesis testing: {T}he
		choice of a null hypothesis.} {\em {J. Am. Stat. Assoc.}},
	99(465):96--104, 2004.
	\MR{2054289}
	
	\bibitem{Efron2007}
	B.~Efron. Size, power and false discovery rates. {\em Ann.
		Statist.}, 35(4):1351--1377, 2007.
	\MR{2351089}
	
	\bibitem{Efron2008}
	B.~Efron. Microarrays, empirical {B}ayes and the two-groups model.
	{\em Statist. Sci.}, 23(1):1--22, 2008.
	\MR{2431866}
	
	\bibitem{ETST2001}
	B.~Efron, R.~Tibshirani, J.~D. Storey, and V.~Tusher. Empirical
	{B}ayes analysis of a microarray experiment. {\em J. Amer.
		Statist. Assoc.}, 96(456):1151--1160, 2001.
	\MR{1946571}
	
	\bibitem{GS2018}
	D.~Gerard and M.~Stephens. {Empirical {B}ayes shrinkage and false
		discovery rate estimation, allowing for unwanted variation}.
	{\em Biostatistics}, 07 2018.
	\MR{4043843}
	
	\bibitem{Ign2016}
	N.~Ignatiadis, B.~Klaus, J.~Zaugg, and W.~Huber. Data-driven
	hypothesis weighting increases detection power in genome-scale multiple
	testing. {\em Nature methods}, 13:577--580, 05 2016.
	
	\bibitem{javanmard2019false}
	A.~Javanmard, H.~Javadi, et~al. False discovery rate control
	via debiased lasso. {\em Electronic Journal of Statistics},
	13(1):1212--1253, 2019.
	\MR{3935848}
	
	\bibitem{JY2016}
	W.~Jiang and W.~Yu. {Controlling the joint local false discovery
		rate is more powerful than meta-analysis methods in joint analysis of summary
		statistics from multiple genome-wide association studies}. {\em Bioinformatics},
	33(4):500--507, 12 2016.
	
	\bibitem{JS04}
	I.~M. Johnstone and B.~W. Silverman. Needles and straw in haystacks:
	{E}mpirical {B}ayes estimates of possibly sparse sequences.
	{\em Ann. Statist.}, 32(4):1594--1649, 2004.
	\MR{2089135}
	
	\bibitem{js05}
	I.~M. Johnstone and B.~W. Silverman. Ebayes{T}hresh: {R} {P}rograms
	for {E}mpirical {B}ayes {T}hresholding. {\em Journal of Statistical
		Software}, 12(8), 2005.
	\MR{2364426}
	
	\bibitem{lee2016}
	N.~Lee, A.-Y. Kim, C.-H. Park, and S.-H. Kim. An improvement
	on local {FDR} analysis applied to functional {MRI} data. {\em Journal
		of neuroscience methods}, 267:115--125, 2016.
	
	\bibitem{LF2016}
	A.~Li and R.~F. Barber. Multiple testing with the structure-adaptive
	{B}enjamini-{H}ochberg algorithm. {\em J. R. Stat. Soc. Ser.
		B. Stat. Methodol.}, 81(1):45--74, 2019.
	\MR{3904779}
	
	\bibitem{Liu2013gaussian}
	W.~Liu. Gaussian graphical model estimation with false
	discovery rate control. {\em The Annals of Statistics}, 41(6):2948--2978,
	2013.
	\MR{3161453}
	
	\bibitem{muelleretal04}
	P.~M\"{u}ller, G.~Parmigiani, C.~Robert, and J.~Rousseau. Optimal
	sample size for multiple testing: {T}he case of gene expression microarrays.
	{\em J. Amer. Statist. Assoc.}, 99(468):990--1001, 2004.
	\MR{2109489}
	
	\bibitem{RRJW2020}
	M.~Rabinovich, A.~Ramdas, M.~I. Jordan, and M.~J. Wainwright.
	Optimal rates and trade-offs in multiple testing.
	{\em Statistica Sinica}, 30:741--762, 2020.
	\MR{4214160}
	
	\bibitem{RRV2019}
	T.~Rebafka, E.~Roquain, and F.~Villers. Graph inference with
	clustering and false discovery rate control. 2019.
	Arxiv preprint \arxivurl{1907.10176}.
	
	\bibitem{RW2009}
	E.~Roquain and M.~van~de Wiel. Optimal weighting for false discovery
	rate control. {\em Electron. J. Stat.}, 3:678--711, 2009.
	\MR{2521216}
	
	\bibitem{Salomond17}
	J.-B. Salomond. Risk quantification for the thresholding rule
	for multiple testing using Gaussian scale mixtures. 2017.
	Arxiv preprint \arxivurl{1711.08705}.
	
	\bibitem{sarkaretal08}
	S.~K. Sarkar, T.~Zhou, and D.~Ghosh. A general decision theoretic
	formulation of procedures controlling {FDR} and {FNR} from a {B}ayesian
	perspective. {\em Statist. Sinica}, 18(3):925--945, 2008.
	\MR{2440399}
	
	\bibitem{Steph2016}
	M.~Stephens. {False discovery rates: {A} new deal}.
	{\em Biostatistics}, 18(2):275--294, 10 2016.
	\MR{3824755}
	
	\bibitem{Storey2003}
	J.~D. Storey. The positive false discovery rate: {A} {B}ayesian
	interpretation and the {$q$}-value. {\em Ann. Statist.},
	31(6):2013--2035, 2003.
	\MR{2036398}
	
	\bibitem{SS2018}
	L.~Sun and M.~Stephens. Solving the empirical {B}ayes normal
	means problem with correlated noise. 2018. Arxiv preprint
	\arxivurl{1812.07488}.
	
	\bibitem{SC2007}
	W.~Sun and T.~T. Cai. Oracle and adaptive compound decision rules
	for false discovery rate control. {\em J. Amer. Statist.
		Assoc.}, 102(479):901--912, 2007.
	\MR{2411657}
	
	\bibitem{SC2009}
	W.~Sun and T.~T. Cai. Large-scale multiple testing under dependence.
	{\em J. R. Stat. Soc. Ser. B Stat. Methodol.}, 71(2):393--424,
	2009.
	\MR{2649603}
	
	\bibitem{zablocki2014}
	R.~W. Zablocki, A.~J. Schork, R.~A. Levine, O.~A. Andreassen, A.~M. Dale,
	and W.~K. Thompson. Covariate-modulated local false discovery
	rate for genome-wide association studies. {\em Bioinformatics},
	30(15):2098--2104, 2014.
\end{thebibliography}

\pagebreak

\appendix

\section{Auxiliary results}\label{sec:AuxiliaryLemmas}
\begin{lemma}
\label{lemmaBoundingqvalue}
Recall the definition \cref{eqn:def:chi} of $\chi$ are recall we write $r(w,t)=wt(1-w)^{-1}(1-t)^{-1}$. For universal constants $c,c'>0$, for all $t\in(0,1)$, there exists $\omega_0(t)$ such that for 
	$w\leq \omega_0(t)$, 
	\begin{align}
		\tilde{m}(w)\left(1+ c \frac{\log\log (1/w)}{\log (1/w)}\right) \leq 2 \overline{G}(\chi(r(w,t))) \leq \tilde{m}(w)\left(1+ c' \frac{\log\log (1/w)}{\log (1/w)}\right).
	\end{align} 
\end{lemma}

\begin{proof}[Proof of Lemma~\ref{lemmaBoundingqvalue}]
	The proof relies on the following inequalities (see \cref{lem:ResultsFromCR18}): for universal constants $C_1,C_2>0$ and $w$ small enough,
	\begin{align}
	2\overline{G}(\zeta(w)) (1-C_2 \zeta(w)^{-3} ) \leq	\tilde m(w) \leq C_1 \zeta(w)^{-3} +2\overline{G}(\zeta(w)).\label{eq:uppermtilde}
	\end{align}
     Let us now prove the lower bound. By 	\cref{lem:ResultsFromCR18}, for a universal constant $c_1>0$, and $w$ small enough (smaller than a threshold that might depend on $t$),
	$
	\zeta(w)- \chi(r(w,t)) \geq  c_1 \frac{\log\log(1/w)}{ \zeta(w)}.
	$
	Hence, since $g$ in nonincreasing on a vicinity of $+\infty$, we have for $w$ small enough
	\begin{align*}
		\overline{G}(\chi(r(w,t)))-\overline{G}(\zeta(w)) 
		&=\int_{\chi(r(w,t))}^{\zeta(w)} g(u)du\\
		&\geq  \left(\zeta(w)- \chi(r(w,t))\right) g(\zeta(w))\\
		&\geq c'_1 \frac{\log\log(1/w)}{ \zeta^3(w)},
	\end{align*}
	for a universal constant $c'_1>0$.
	Combining the last display with \eqref{eq:uppermtilde} leads to
	\begin{align*}
		\tilde m(w) &\leq C \zeta(w)^{-3} +2\overline{G}(\chi(r(w,t))) - 2c'_1 \frac{\log\log(1/w)}{ \zeta^3(w)}\\
		&\leq 2\overline{G}(\chi(r(w,t))) - c'_1 \frac{\log\log(1/w)}{ \zeta^3(w)},
	\end{align*}
	for $w$ small enough. The lower bound now follows from  $\tilde{m}(w)\asymp 1/\zeta(w)$ and $\zeta(w)\asymp (\log (1/w))^{1/2}$ (see \cref{lem:ResultsFromCR18}).
	
	For the upper bound part, we proceed similarly:  let us first prove that, for an universal constant $c_2>0$, for $w$ small enough (smaller than a threshold that might depend on $t$),
	\begin{align}
		\zeta(w)- \chi(r(w,t)) \leq  c_2 \frac{\log\log(1/w)}{ \zeta(w)}.\label{eq:zetachi}
	\end{align}
	This comes from \cref{lem:ResultsFromCR18}: for $w$ small enough,
	\begin{multline*}
		\zeta(w)^2- \chi(r(w,t))^2
		\leq 2\log (1/w) + 2 \log\log (1/w)  - 2\log ((1-w)(1-t)/(tw))\\  + \log(\log((1-w)(1-t)/(tw)))  +C+C'
		\leq 4 \log\log (1/w).
	\end{multline*}
	 This leads to \eqref{eq:zetachi}. 
	Now, proceeding as for the lower bound, we have
	\begin{align*}
		\overline{G}(\chi(r(w,t)))-\overline{G}(\zeta(w)) 
		&=\int_{\chi(r(w,t))}^{\zeta(w)} g(u)du\\
		&\leq  \left(\zeta(w)- \chi(r(w,t))\right) g(\chi(r(w,t)))\\
		&\leq c'_2 \frac{\log\log(1/w)}{ \zeta^3(w)},
	\end{align*}
	Combining the latter with \eqref{eq:uppermtilde} gives
	\begin{align*}
		2\overline{G}(\chi(r(w,t))  \leq \tilde m(w) (1-C_2 \zeta(w)^{-3} )^{-1} +  2c'_2 \frac{\log\log(1/w)}{ \zeta^3(w)},
	\end{align*}
	which implies the upper bound. \qedhere

\end{proof}

\begin{lemma}\label{lem:w+/w-to1}
	Define $w_\pm$ as in \cref{eqn:def:w-,eqn:def:w+}, define $\lambda_+$ as in \cref{eqn:def:lambda+}, and recall the definitions \cref{eqn:def:nu,eqn:def:rho} of $\nu_n,\rho_n$ and \cref{eqn:def:r} of $r$. Then
	\begin{equation*}\label{eqn:w+/w-to1} w_+/w_- - 1 =O(\max(\nu_n,\rho_n)), \quad \frac{r(w_+,\lambda_+)}{r(w_-,\lambda_+)}-1 = O(\max(\nu_n,\rho_n)).\end{equation*} 
\end{lemma}
\begin{proof}
We have $w_+\geq w_-$ (\cref{lem:ExistsW+-}), hence we focus on bounding $w_+/w_- -1$ from above. Since \cref{lem:ExistsW+-} also implies that $\log(1/w_-)\asymp \log(1/w_+)\asymp \log(n/s_n)$, we use \cref{lem:m1boundsAnyBoundary} to bound $m_1$ in the definitions \cref{eqn:def:w-,eqn:def:w+} of $w_-$ and $w_+$, and deduce that \begin{align*} (1-\nu_n)(n-s_n)w_+\tilde{m}(w_+)&\leq s_n,\\ (1+\nu_n)(n-s_n)w_-\tilde{m}(w_-)&\geq s_n (1-\rho_n).\end{align*}
Taking the ratio, we deduce that
\[ \frac{w_+\tilde{m}(w_+)}{w_-\tilde{m}(w_-)}\leq (1-\nu_n)^{-1}(1+\nu_n)(1-\rho_n)^{-1}.\]
Then, since $w_+\geq w_-$ and $\tilde{m}$ is increasing, we see that
\[ \frac{w_+}{w_-}-1\leq \frac{w_+\tilde{m}(w_+)}{w_-\tilde{m}(w_-)}-1= O(\max(\nu_n,\rho_n)),\] as claimed.

Finally, since $w_+,w_-\to 0$, we deduce that \[\frac{1-w_-}{1-w_+}-1 = \frac{(w_+/w_-)-1}{(1-w_+)/w_-} = o(w_+/w_- - 1),\] hence
	\[ \frac{r(w_+,\lambda_+)}{r(w_-,\lambda_+)}-1 = O\brackets[\Big]{\max\brackets[\Big]{ \frac{w_+}{w_-}-1,\frac{1-w_-}{1-w_+}-1}}= O(\max(\nu_n,\rho_n)).\qedhere \] 
\end{proof}

\begin{lemma}\label{lem:xi+-differencebounded}
Define $w_\pm,\lambda_\pm,\xi,r$ as in \cref{eqn:def:w-,eqn:def:w+,eqn:def:lambda-,eqn:def:lambda+,eqn:def:xi,eqn:def:r}. Then
\[\xi(r(w_-,\lambda_+))^2-\xi(r(w_+,\lambda_+))^2 = O(1)\]
\end{lemma}
\begin{proof}
Write $r_\pm=r(w_\pm,\lambda_+)$ and $\xi_\pm=\xi(r_\pm)$. \Cref{lem:ResultsFromCR18} gives us the near matching upper and lower bounds on $\xi$ that for $u\in(0,1)$ small enough we have 
\begin{align*} \xi(u) &\leq (2\log(1/u)+2\log\log(1/u)+6\log 2)^{1/2}, \\
\xi(u) &\geq (2\log(1/u)+2\log\log(1/u) + 2\log 2)^{1/2}.\end{align*}
Using these bounds and monotonicity of $\xi:=(\phi/g)^{-1}$ (which follows from the fact that of $\phi/g$ is decreasing on $x\geq 0$ as in \cref{lem:monotonicity}) we deduce that
\begin{equation}\label{eqn:xi_-^2-xi_+^2}
	0\leq \xi_-^2-\xi_+^2  \leq 2\log\brackets[\Big]{\frac{r_+}{r_-}} + 2\log \log (1/r_-)-2\log\log(1/r_+) +4\log 2.
\end{equation}
Observe that
\[ \log\log(1/r_-)-\log\log(1/r_+) = \log\brackets[\Big]{\frac{\log(1/r_-)}{\log(1/r_+)}}=\log\brackets[\Big]{1+\frac{\log(r_+/r_-)}{\log(1/r_+)}}.\]
Using the standard bound $\log(1+x)\leq x$ for $x>-1$ and the fact that $r_+\to 0$ (by \cref{lem:ExistsW+-,lem:ExistsLambda+-}), this last expression is upper bounded by \[ \frac{\log(r_+/r_-)}{\log(1/r_+)}=o(\log(r_+/r_-)),\]
and, using \cref{lem:w+/w-to1}, we similarly have
\[ \log(r_+/r_-)\leq \frac{r_+}{r_-}-1= O(\max(\nu_n,\rho_n))=o(1).\] 
Inserting into \cref{eqn:xi_-^2-xi_+^2} we see that $\xi_-^2-\xi_+^2=O(1)$,
as claimed.
\end{proof}

\begin{lemma} \label{lem:m1boundsAnyBoundary}
There exist constants $\omega_0\in (0,1)$ and $c,c'>0$ such that for any sequence $s_n/n\to 0$ and $v_n\to \infty$, 
		for all $\theta_0\in \ell_0(s_n,v_n)$, for any $i$ such that $\theta_{0,i}\neq 0$, we have for any $w\in[s_n/n,\omega_0]$,
		\begin{align} \label{eqn:m1bounds} 
			(1-\rho_n)w^{-1} \leq m_1(\theta_{0,i},w ) &\leq w^{-1}, \\ \label{eqn:mtildebounds} c(\log (1/w))^{-1/2}\leq \tilde{m}(w)&\leq c'(\log (1/w))^{-1/2},\end{align}
		where we recall that $\rho_n=e^{-v_n^2/9}$ as in \cref{eqn:def:rho}.	
\end{lemma}

\begin{proof}
\Cref{lem:ResultsFromCR18} tells us that $\tilde{m}(w)\asymp \zeta(w)^{-1}$ and $\zeta(w)\sim (2\log(1/w))^{-1/2}$, yielding \cref{eqn:mtildebounds}.
It also tells us, regarding $m_1$,  that there exists $c_1>0$ such that for all $x\in\RR$ and all $w\in(0,1],$ \begin{equation} \label{eqn:m_1bound} m_1(x,w)\leq \min(w,c_1)^{-1},\end{equation} 
	so that the upper bound in \cref{eqn:m1bounds} is immediate upon choosing $\omega_0=\min(c_1,1)$. 
	
	It remains to show the lower bound on $m_1$. This lower bound is a sharpening of Lemma~S-29 in \cite{CR18} and is proved similarly. By assumption, if, for some $i$, $\abs{\theta_{0,i}}\neq 0$, then we may assume by symmetry of $m_1$ that $\mu=\theta_{0,i}>0$ and we further have
	\[ \mu \geq \sqrt{2\log(n/s_n)}+v_n.\]
	Writing $p=p(n,w)=\frac{v_n}{\zeta(w)}$ and 	$a=1+0.5p$, using monotonicity of $\phi/g$ and hence of $\beta$ (\cref{lem:monotonicity}),
		we have for $w$ such that $w\abs{\beta(0)}<1/2$,  
		\begin{align*}
			w m_1(\mu,w) 
			& =  \int_{\abs{x}>a\zeta(w)} \frac{w \beta(x)}{1+w\be(x)} \phi(x-\mu)dx +\ \int_{-a\zeta(w)}^{a\zeta(w)} \frac{w \beta(x)}{1+w\be(x)} \phi(x-\mu)dx \\
			&\geq \int_{x>a\zeta(w)} \frac{w\beta(x)}{1+w\be(x)} \phi(x-\mu)dx -   \int_{-a\zeta(w)}^{a\zeta(w)} \phi(x-\mu)dx\\
			&\geq  \frac{w\beta(a\zeta(w))}{1+w\be(a\zeta(w))} \overline{\Phi}(a\zeta(w)-\mu) - (1-\overline{\Phi}(a\zeta(w)-\mu)).
		\end{align*}
		Increasingness of $\beta$ implies that $\zeta$ is decreasing, so that also using \cref{lem:ResultsFromCR18} and a Taylor expansion, we have, for some $\Delta_n\to 0$, \[\begin{split} a\zeta(w)-\mu\leq \zeta(w) -\sqrt{2\log(n/s_n)}- 0.5 v_n &\leq \zeta(s_n/n) -\sqrt{2\log(n/s_n)}- 0.5 v_n \\ &\leq  \Delta_n -0.5 v_n.\end{split}\] 
		By standard properties of $\bar{\Phi}$, including the tail bound $\bar{\Phi}(x)\asymp \phi(x)/x$, 
		\[1-\bar{\Phi}( \Delta_n -0.5 v_n)=\bar{\Phi}(    0.5 v_n-\Delta_n)\ll e^{-(0.5 v_n-\Delta_n)^2/2} \leq e^{-(0.5 v_n)^2 /2} e^{v_n\Delta_n/2 }  \ll \rho_n.\]
		In particular, we have, for $n$ large,
		\[1-\bar{\Phi}(a\zeta(w)-\mu)\leq 1-\bar{\Phi}(\Delta_n-0.5v_n)\leq \rho_n/3.\] Additionally,  $w\beta(a\zeta(w))=\beta(a\zeta(w))/\beta(\zeta(w))=((g/\phi)(a\zeta(w))-1)/((g/\phi)(\zeta(w))-1)$ tends quickly to infinity:
		\[
		w\beta(a\zeta(w))\gtrsim  \frac{g(a\zeta(w))}{g(\zeta(w))}  \:\frac{\phi(\zeta(w))}{\phi(a\zeta(w))} 
		\gtrsim  \frac{\phi(\zeta(w))}{\phi(a\zeta(w))} = e^{ (a^2-1) \zeta(w)^2/2 } \gg e^{ v_n^2 0.5^2/2 }\gg \rho_n^{-1}.
		\]
		In particular, we see that, for $n$ large, \[ \frac{w\beta(a\zeta(w))}{1+w\beta(a\zeta(w))} = 1- \frac{1}{1+w\beta(a\zeta(w))} \geq 1- \frac{1}{w\beta(a\zeta(w))} \geq 1-\rho_n/3.\]  Inserting these bounds we find that \[ wm_1(\mu,w)\geq (1-\rho_n/3)(1-\rho_n/3) -\rho_n/3 \geq 1- \rho_n. \qedhere\]
\end{proof}

The following two technical lemmas give precise bounds on $F_w(\lambda)$ and on $E_{\theta_0=0}[\ell_{1,w} \mid \ell_{1,w'}<\lambda]$ for suitable $w,w'$ and $\lambda$ which are essential to obtaining the convergence rates in \cref{thm:FDRConvergenceWithRate}.

\begin{lemma}\label{lem:AsymptoticsOfFwLambda}
	The function $F_w(\lambda)=P_{\theta_{0,i}=0}(\ell_{i,w}\leq \lambda)$ is continuous and strictly increasing in $\lambda$. Assume that $w=w_n$ and $\lambda=\lambda_n\in (0,1)$ satisfy $\lambda\to 1$ and $w/(1-\lambda)\to 0$. Then
	\[ F_w(\lambda)=2\bar{\Phi}(\xi(r(w,\lambda)))\asymp w(1-\lambda)^{-1} (\log( (1-\lambda)/w) )^{-3/2}\] as $n\to \infty$, where $\xi=(\phi/g)^{-1}$ and $r(w,t)=w(1-w)^{-1}t(1-t)^{-1}$.
	If in fact $w^c/(1-\lambda)\to 0$ for some $c<1$ then 
		\[ F_w(\lambda)\asymp w(1-\lambda)^{-1} (\log(1/w))^{-3/2}.\]
	
\end{lemma}
\begin{proof}
	A direct calculation, as needed also in proving \cref{lem:monotonicity}, yields
	\begin{equation}\label{eqn:LvalsXi} \ell_{i,w}(X)\leq t \iff \abs{X_i}\geq \xi(r(w,t)), \end{equation} so that $F_{w}(x)=2\bar{\Phi}(\xi(r(w,x)))$ as claimed and hence $F_w$ is continuous. 
	
	Next, we use a standard Gaussian tail bound, 
	 the definition of $\xi$, the definition \cref{eqn:def:gInQuasiCauchy} of $g$ in the quasi-Cauchy case, the fact that $r(w,\lambda)\asymp w/(1-\lambda)$ as $w\to 0$ and $\lambda\to 1$, and the fact that $\xi(u)\asymp (\log(1/u))^{1/2}$ as $u\to 0$ (see \cref{lem:ResultsFromCR18}) to see that	\[ \begin{split} \bar{\Phi}(\xi(r(w,\lambda))\asymp \frac{\phi(\xi(r(w,\lambda))}{\xi(r(w,\lambda)}\asymp \frac{r(w,\lambda) g(\xi(r(w,\lambda))}{\xi(r(w,\lambda))} &\asymp r(w,\lambda)\xi(r(w,\lambda))^{-3} \\ &\asymp \frac{ w}{1-\lambda} \brackets[\Big]{\log\brackets[\Big]{\frac{1-\lambda}{w}}}^{-3/2},\end{split} \] as claimed.
 Note that $\log((1-\lambda)/w)\leq \log(1/w)$, and that when $w^c/(1-\lambda)\to 0$ we have $\log ((1-\lambda)/w)\gtrsim \log( 1/w^{1-c})\asymp \log(1/w). $
\end{proof}

\begin{lemma}\label{lem:ExpectationsSlowlyTo1}
	Suppose for sequences $w_1=w_{1,n},w_2=w_{2,n}$ and $\lambda=\lambda_n$ taking values in $[0,1]$ that $\lambda\to1$, that both $w_2/w_1$ and $w_1/w_2$ are bounded, and that $w_1^c/(1-\lambda)\to 0$ for some $c<1$. Then
%
	\begin{equation*}
	1-E_{\theta_0=0}[\ell_{1,w_1}(X) \mid   \ell_{1,w_2}(X)<\lambda]\asymp
	(1-\lambda) \log(1/(1-\lambda)),
\end{equation*}	
%
Let us also note here that for fixed $w_1,w_2$, $E_{\theta_0=0}[\ell_{1,w_1}(X)\mid \ell_{1,w_2}(X) <\lambda]$ is continuous in $\lambda$.
\end{lemma}


\begin{proof}
	Recall the definitions $\beta(x)=\tfrac{g}{\phi}(x)-1$, $\zeta(w)=\beta^{-1}(1/w)$, $\xi=(\phi/g)^{-1}$, and recall that $\ell_{1,w}(X)<\lambda$ if and only if $|X_1|> \xi(r(w,\lambda))$, see \eqref{eqn:LvalsXi}. Using symmetry of the densities $\phi$ and $g$ we see that for all $w_1,w_2\in (0,1)$,
	\[
	E_{\theta_0=0}[\ell_{1,w_1}(X)\:|\: \ell_{1,w_2}(X)<\lambda]= \frac{\int_{\xi_{w_2}}^\infty \frac{(1-w_1)\phi(x)}{(1-w_1)\phi(x)+w_1g(x)}\phi(x)\dx}{\bar{\Phi}({\xi_{w_2}})}, \]
where we have introduced the notation $\xi_{w_2}:=\xi(r(w_2,\lambda))$. The expression on the right is continuous at any $\lambda$ such that the denominator is bounded away from zero, i.e.\ at any $\lambda\neq 0$, hence the same is true of the conditional expectation.

	Write $h_{w_1}(x)=w_1 \beta(x)\phi(x)/(1+w_1\beta(x))$. For $w_1,w_2$ small enough, the following bounds hold: 
	\begin{align*}
		\phi(x)/2 & \leq h_{w_1}(x)\leq \phi(x), && x\in [\zeta(w_1),\infty);\\
		w_1 g(x)/4 & \leq h_{w_1}(x) \leq w_1 g(x) && x\in [{\xi_{w_2}}, \zeta(w_1)].
	\end{align*}
	To obtain these inequalities we have used monotonicity of $\phi/g$ and hence $\beta$, and the fact that $\beta(\zeta(w))=1/w$. The first inequalities then follow from the expression $h_{w_1}(x)= \brackets{\frac{w_1\beta(x)}{1+w_1\beta(x)}}\phi(x)$, while the latter inequalities result from the expression $h_{w_1}(x)=w_1 g(x) \brackets{\frac{1-(\phi/g)(x)}{1+w_1\beta(x)}}$ and the fact that $(\phi/g)(x)\leq 1/2$ for $x$ large enough. 
	By assumption there exists $C>0$ such that $w_1\leq Cw_2$ for all $n$ large enough, and note that also $\lambda\geq C/(C+1)$ by further increasing $n$ if necessary. Recalling the relationship \cref{eqn:zeta-xi-relationship} and using that decreasingness of $\phi/g$ (\cref{lem:monotonicity}) implies the same of $\xi=(\phi/g)^{-1}$, we then have 
	\[\zeta(w_1) = \xi(w_1/(1+w_1)) \geq \xi(w_1) \geq \xi(C w_2) \geq \xi(r(w_2,\lambda))= \xi_{w_2}.\]  In addition, since $g$ is decreasing for $x$ large, we have 
	\[w_1 g(\zeta(w_1))/4 \leq h_{w_1}(x) \leq w_1 g(\xi_{w_2}), \qquad x\in [\xi_{w_2},\zeta(w_1)].
	\] 
	
	Then
	\begin{align*} 
		&\int_{\xi_{w_2}}^\infty \frac{\phi(x)}{(1-w_1)\phi(x)+w_1g(x)}\phi(x)\dx
		\\=& \int_{\xi_{w_2}} ^\infty \frac{1}{1+w_1\beta(x)} \phi(x) \dx
		\\ =& \int_{\xi_{w_2}}^\infty \phi(x)\dx -\int_{\xi_{w_2}}^\infty h_{w_1}(x) \dx  \\ 
		=& \bar{\Phi}({\xi_{w_2}})-\int_{{\xi_{w_2}}}^{\zeta(w_1)} h_{w_1}(x)\dx -\int_{\zeta(w_1)}^\infty h_{w_1}(x)\dx \\
		\geq& \bar{\Phi}(\xi_{w_2})- (\zeta(w_1)-\xi_{w_2}) w_1 g(\xi_{w_2}) - \bar{\Phi}(\zeta(w_1)). 
	\end{align*} We can similarly upper bound the integral, so we deduce the inequalities 
	\begin{equation}\label{eqn:ExpectationUpperLowerBound}
		\begin{split}
			&\frac{(1-w_1)}{\bar{\Phi}(\xi_{w_2})} \sqbrackets[\Big]{ \bar{\Phi}(\xi_{w_2})-\bar{\Phi}(\zeta(w_1)) - (\zeta(w_1)-\xi_{w_2})w_1 g(\xi_{w_2})} \\
			& \quad \leq E_{\theta_0=0}[\ell_{1,w_1}(X)\:|\: \ell_{1,w_2}(X)<\lambda] \\
			& \qquad \leq \frac{(1-w_1)}{\bar{\Phi}(\xi_{w_2})} \sqbrackets[\Big]{\bar{\Phi}(\xi_{w_2}) - \frac{1}{2} \bar{\Phi}(\zeta(w_1)) -\frac{1}{4} (\zeta(w_1)-\xi_{w_2}) w_1 g(\zeta(w_1)) }.
		\end{split} 
	\end{equation}
	Now, let us study in detail the order of each term. First, usual normal tail bounds, the definition of $\zeta$, the definition \cref{eqn:def:gInQuasiCauchy} of $g$ in the quasi-Cauchy case and Lemma~\ref{lem:ResultsFromCR18} (which tells us that $\zeta(w)^2\asymp \log (1/w)$) imply that for $w_1$ small enough
	\[\bar{\Phi}(\zeta(w_1))\asymp \frac{\phi(\zeta(w_1))}{\zeta(w_1)} \asymp w_1 \frac{g(\zeta(w_1))}{\zeta(w_1)} \asymp w_1 \zeta(w_1)^{-3}\asymp w_1 (\log (1/w_1))^{-3/2}.
	\] 
Similarly to the proof of \cref{lem:AsymptoticsOfFwLambda}, observe that $(1-\lambda)/w_2\geq w_1^{c-1}$ and for $n$ large and hence that $\log((1-\lambda)/w_2)\asymp \log(1/w_1)$.
Using the definition of $\xi$ and \cref{lem:ResultsFromCR18} (which tells us that $\xi(u)^2\asymp \log(1/u)$), we then obtain
	\[ \begin{split}
	\bar{\Phi}({\xi_{w_2}}) \asymp  \frac{\phi(\xi_{w_2})}{\xi_{w_2}} = r(w_2,\lambda) \frac{g(\xi_{w_2})}{\xi_{w_2}} \asymp \frac{w_2}{1-\lambda} \xi_{w_2}^{-3} &\asymp \frac{w_2}{1-\lambda} \brackets[\Big]{\log
		\brackets[\Big]{\frac{1-\lambda}{w_2}}}^{-3/2} \\ &\asymp \frac{w_1}{1-\lambda}\brackets[\big]{ \log(1/w_1)}^{3/2} 
\end{split}	\]
We deduce that $0\leq \bar{\Phi}(\zeta(w_1))/\bar{\Phi}(\xi_{w_2})\lesssim 1-\lambda$.

We apply \cref{lem:zeta-xi} with $w=w_1$ and with $t\in (0,1)$ such that $r(w_1,t)=r(w_2,\lambda)$. Observing that 	\[ \frac{1}{1-t}= r(w_1,t) \frac{1-w_1}{tw_1}= r(w_2,\lambda)\frac{1-w_1}{tw_1}\asymp \frac{\lambda}{1-\lambda} \asymp \frac{1}{1-\lambda},\] so that $w_1^c/(1-t)\to 0$, we deduce that
\[	 \zeta(w_1)-\xi_{w_2}\asymp \frac{\log(1/(1-t))}{(\log(1/w_1))^{1/2}}\asymp  \frac{\log(1/(1-\lambda))}{(\log(1/w_1))^{1/2}}.\]
Again using that $\log((1-\lambda)/w_2)\asymp \log(1/w_1)$, it follows that 
	\[
(\zeta(w_1)-\xi_{w_2}) w_1 g(\zeta(w_1))\asymp 	(\zeta(w_1)-\xi_{w_2}) w_1 g(\xi_{w_2})\asymp 
 (1-\lambda)\log(1/(1-\lambda)) \bar{\Phi}({\xi_{w_2}}),
	\]
	since we showed above that $ w_1(1-\lambda)^{-1} (\log
	(1/w_1))^{-3/2} \asymp \bar{\Phi}({\xi_{w_2}})$.
	Feeding these bounds into \eqref{eqn:ExpectationUpperLowerBound} yields that for some $c_1,c_2,c_3>0$
	\begin{align*} 1-E_{\theta_0}[\ell_{1,w_1} \mid \ell_{1,w_2}<\lambda] &\geq w_1 +c_1(1-w_1)((1-\lambda)\log(1/(1-\lambda)),\\ 
		1-E_{\theta_0}[\ell_{1,w_1} \mid \ell_{1,w_2}<\lambda] &\leq w_1 + c_2(1-\lambda) +c_3((1-\lambda)\log(1/(1-\lambda)). \end{align*} 
	The lower bound follows upon discarding the term $w_1$ and noting that $1-w_1\geq 1/2$ for $n$ large; for the upper bound we note that $w_1+c_2(1-\lambda)=o\brackets[\big]{(1-\lambda)\log(1/(1-\lambda))}.$ 
\end{proof}

\begin{lemma}\label{lem:zeta-xi} Suppose for sequences $w=w_n$ and $t=t_n$ taking values in $[0,1]$ that $t\to 1$ and $w^c/(1-t)\to 0$ for some $c<1$. 
	Then $\zeta(w)-\xi(r(w,t))\geq 0$ for $n$ large enough and, as $n\to \infty$, 
		\begin{align}
			\zeta(w)-\xi(r(w,t))\asymp \frac{\log(1/(1-t))}{(\log (1/w))^{1/2}}.
	\end{align}
\end{lemma}

\begin{proof}
	For $1-t\leq 1/2$ and $w/(1-t)\leq 0.5$, we have $r(w,t)=\frac{tw}{(1-t)(1-w)}\geq \frac{wt}{1-t} \geq 0.5 \frac{w}{1-t}$, so that $\log\log (1/r(w,t))\leq \log(\log(2)+\log((1-t)/w)\leq \log(2)+\log\log((1-t)/w) $.
	Hence, using bounds on $\zeta$ and $\xi$ from \cref{lem:ResultsFromCR18} and noting that $\log(1/w)\geq \log((1-t)/w)$ and that $\log(1/(1-w))$ is bounded, we see that for $1-t$ and $w/(1-t)$ small enough  we have for constants $c,c'$ 
	\begin{align*} &
		\zeta(w)^2-\xi(r(w,t))^2 \\ \geq& 2\log \brackets[\big]{\tfrac{1}{w}}+2\log \log \brackets[\big]{\tfrac{1}{w}} +c -  \brackets[\Big]{2\log\brackets[\big]{\tfrac{1}{r(w,t)}} + 2\log\log \brackets[\big]{\tfrac{1}{r(w,t)}} +6\log 2}\\
		\geq& 2 \log(t/(1-t)) +2\log( \log(1/w) / \log ((1-t)/w)) +c'\\
		\geq& 2 \log(1/(1-t)) +c'\\
		\geq& \log(1/(1-t)).
	\end{align*}
	Conversely, for $w\leq 1/2$, we have $r(w,t)=\frac{tw}{(1-t)(1-w)}\leq 2 \frac{w}{1-t}$, so that \[\log\log (1/r(w,t))\geq \log(\log(0.5)+\log((1-t)/w)\geq \log(0.5)+\log\log((1-t)/w),\] provided $n$ is large enough that $\log((1-t)/w)+\log(0.5)\geq 0.5\log((1-t)/w)$. Note also, as in the proof of \cref{lem:AsymptoticsOfFwLambda} that the condition on $w_1^c/(1-\lambda)\to 0$ implies that $\log(1/w)/\log((1-t)/w)$ is bounded. Again using bounds on $\zeta$ and $\xi$ from \cref{lem:ResultsFromCR18}, 
	for $1-t$ and $w/(1-t)$ small enough we deduce that for constants $C,C',C''$ we have
		\begin{align*} & \zeta(w)^2-\xi(r(w,t))^2  \\ \leq &2\log \brackets[\big]{\tfrac{1}{w}}+2\log \log\brackets[\big]{\tfrac{1}{w}}+C -  \brackets[\Big]{2\log\brackets[\big]{\tfrac{1}{r(w,t)}} + 2\log\log \brackets[\big]{\tfrac{1}{r(w,t)}} +2\log 2}\\
		\leq & 2 \log(t/(1-t)) +2\log( \log(1/w) / \log ((1-t)/w)) +C'\\
		\leq &2 \log(1/(1-t))+C''.
	\end{align*}
 This entails
		\[
			\zeta(w)-\xi(r(w,t)) \asymp \frac{\log(t/(1-t))}{\zeta(w)+\xi(r(w,t))},
		\]
	and the result thus follows from
		$
		\zeta(w)\leq \zeta(w)+\xi(r(w,t))\leq 2\zeta(w)
		$ (the latter being implied by the above calculations)
		and $\zeta(w)\asymp (\log (1/w))^{1/2}$.
\end{proof}

\begin{lemma} \label{lem:c1alpha}
	For any $\beta>0$ there exists $c_1=c_1(\beta)>0$ such that for any $s_n>c_1 (\log n)^2/\log \log n$ satisfying $n/s_n\to \infty$ we have for $n$ large enough
	\[ \frac{\log \log (n/s_n)}{\log (n/s_n)} \geq \alpha \brackets[\Big]{\frac{\log s_n}{s_n}}^{1/2}.\]
	Consequently, the conclusions of \cref{thm:FDRConvergenceWithRate} hold upon replacing the assumption $s_n\geq (\log n)^3$ with $s_n\geq b(\log n)^2/\log\log n$ for some large enough $b$.
\end{lemma}
\begin{proof}
	Write \[ p(s)=\frac{\log \log (n/s)}{\log (n/s)}, \quad q(s)=\brackets[\Big]{\frac{\log s}{s}}^{1/2}.\] Since $u^{-1}\log u\leq (u')^{-1}\log u'$ for $u\geq u'\geq e$, we see that $p(s_n)\geq p(1)$ (apply with $u=\log n$, $u'=\log (n/s_n)$, and note $u'>e$ for $n$ large enough since $n/s_n\to \infty$). Similarly we notice that $q$ is decreasing, at least on $s>e$, so that $q(s_-)\geq q(s_n)$ for $s_-:=c_1 (\log n)^2/\log \log n$. It therefore suffices to show that $p(1)\geq \beta q(s_-)$.
	
	Observe that for $n$ large enough we have $s_-\leq (\log n)^2$, hence $\log s_- \leq 2\log \log n.$ It follows that $s_-/\log s_-\geq\frac{1}{2} c_1 (\log n / \log \log n)^2.$ Thus, \[q(s_-)\leq \brackets[\Big]{\frac{2}{c_1}}^{1/2} \frac{\log \log n}{\log n}= (c_1 /2)^{-1/2} p(1).\] The result follows for $c_1=2\beta^2$.
	
	To see that the proof of \cref{thm:FDRConvergenceWithRate} holds under the weaker condition on $s_n$, note that the lower bound on $s_n$ was not assumed for any of the core lemmas, and in the proof of the theorem itself was only used to show that for any $\beta>0$, $\nu_n\leq \beta\eps_n$ for $n$ large enough.
\end{proof}

Finally, for the reader's convenience, we gather some results together whose proofs are omitted because they can be found elsewhere. The following lemma collects results from \cite{CR18}; we remark that while the setting of that paper assumes polynomial sparsity, the results gathered here do not depend on that assumption.
Some of the following results are originally stated with dependence on $g$ and a related parameter $\kappa\in[1,2]$; here, with $g$ explicitly given in \cref{eqn:def:gInQuasiCauchy}, we substitute $\kappa=2$ and use the bounds $\norm{g}_\infty:=\sup_x \abs{g(x)}\leq 1/\sqrt{2\pi}$ and  $x^{-2}/(2\sqrt{2\pi})\leq g(x)\leq x^{-2}/\sqrt{2\pi}$ for $\abs{x}\geq 2$ to simplify expressions.

\begin{lemma}[Results from \cite{CR18}] \label{lem:ResultsFromCR18}
	\begin{lemenum}
			\item Lemma S-10: $\ell_{i,w_-}(X)\geq q_{i,w_-}(X)$.
		\item Lemma S-12: $\xi(u)\sim (2\log (1/u))^{1/2}$, and more precisely, for $u$ small enough, \begin{align*}
			\xi(u) &\geq \brackets[\Big]{2\log (1/u) + 2\log\log(1/u)+2\log 2}^{1/2} \\
			\xi(u) &\leq \brackets[\Big]{ 2\log(1/u) + 2\log\log (1/u) +6\log 2}^{1/2}.
		\end{align*}
		\item Lemma S-14: $\zeta(w)\sim (2\log(1/w))^{1/2}$. More precisely, for  constants $c,C\in\mathbb{R}$ and for $w$ small enough,
	\begin{multline*} (2\log (1/w)+2\log \log(1/w)+c)^{1/2} \\ \leq \zeta(w)\leq (2\log (1/w)+2\log \log(1/w)+C)^{1/2}.
	\end{multline*}
		\item Proof of Lemma S-15: $\zeta(w)- \chi(r(w,t)) \geq  c_1 \frac{\log\log(1/w)}{ \zeta(w)}$  for a universal constant $c_1>0$, for $w$ small enough (smaller than a threshold that might depend on $t$).
	\item Eq.~(S-15): for some constant $C'>0$ and $u\in(0,1]$ small enough,
	\[ \chi(u)\geq 
	\left(2\log (1/u) - \log\log(1/u) - C'\right)^{1/2}.
	\]
		\item Lemma S-20: there exists $c_1>0$ such that for any $x\in\RR$ and $w\in(0,1]$, $\abs{\beta(x)/(1+w\beta(x))}\leq (w \wedge c_1)^{-1}$.
		\item Lemma S-21: there exists $c_1>0$ such that $m_1(x,w)\leq (\min(c_1,w))^{-1}$ for all $x\in \RR$. The function $\tilde{m}$ is continuous, non-negative and increasing. For any fixed $\tau$ the function $w\mapsto m_1(\tau,w)$ is continuous and decreasing.
			\item Lemma S-23: $\tilde{m}(w)\asymp \zeta(w) g(\zeta(w))\asymp \zeta(w)^{-1}$.
	\item Proof of Lemma~S-23: for universal constants $C_1,C_2>0$ and $w$ small enough, $\tilde m(w) \leq C_1 \zeta(w)^{-3} +2\overline{G}(\zeta(w))$ and  $\tilde m(w) \geq 2\overline{G}(\zeta(w)) (1-C_2 \zeta(w)^{-3} )$.
		\item Lemma S-26, Corollary S-28: for $m_2(\theta_{0,i},w)=E_{\theta_0} (\beta(X_i)/[1+w\beta(X_i)])^2 $, there exist constants $C,\omega_0,M_0>0$ such that for all $w\leq \omega_0$ and all $\tau\geq M_0$
	\begin{align*} m_2(0,w)&\leq C\bar{\Phi}(\zeta(w))w^{-2},\\
		m_2(\tau,w)&\leq C m_1(\tau,w)w^{-1}\end{align*}
	\item Lemma S-40: $\bar{\Phi}(x)\sim x^{-1}\phi(x)$ as $x\to \infty$. More precisely,
\[ \frac{x^2}{1+x^2} \frac{\phi(x)}{x} \leq \bar{\Phi}(x)\leq \frac{\phi(x)}{x}.\]
	\end{lemenum}
\end{lemma}
\begin{lemma}[Bernstein's inequality] \label{lem:Bernstein} Let $U_i, i\leq n$ be independent random variables taking values in $[0,1]$. Then, for any $u>0$, 
	\[P\brackets[\Big]{\sum_{i=1}^n (U_i -E [U_i])\geq u}\leq \exp\brackets[\Big]{-\frac{u^2/2}{\sum_{i=1}^n \Var(U_i) + u/3}},\]
	and	  \[P\brackets[\Big]{\abs[\Big]{\sum_{i=1}^n (U_i -E [U_i])}\geq u}\leq 2 \exp\brackets[\Big]{-\frac{u^2/2}{\sum_{i=1}^n \Var(U_i) + u/3}}.\]
\end{lemma}

\section{Notation}\label{sec:notation} 
\begin{description}
	\item[$X$]$=(X_1,\dots,X_n)$ the data, with $X_i=\theta_{i}+\eps_i$, where the $\eps_i$ are i.i.d.\ Gaussians $\eps_i\sim \Nn(0,1)$.  
	\item[$\theta_0$] the unknown true parameter in $\ell_0(s_n,v_n)$.
	\item[$P_{\theta_0}$] the law of $X$ with parameter $\theta_0$, $E_{\theta_0}$ the associated expectation.
	\item[$\ell_0(s)$]$=\braces{\theta\in \RR^\NN : \#\braces{1 \leq i \leq n : \theta_i\neq 0}=\norm{\theta}_{\ell_0}\leq s}$.
	\item[$S_0$]$=\braces{i : \theta_{0,i}\neq 0}$ the support of the vector $\theta_0\in\ell_0(s)$.
	\item[$\ell_0(s_n,v_n)$]$=\braces{\theta \in \ell_0(s_n) : \abs{\theta_i}\geq \sqrt{2\log(n/s_n)}+v_n \text{ for } i\in S_0,~\abs{S_0}=s_n}$, with $s_n\to \infty$, $n/s_n\to \infty$, $v_n\to \infty$. (And $s_n\geq (\log n)^3$, $v_n\geq 3(\log\log (n/s_n))^{1/2}$ for \cref{thm:FDRConvergenceWithRate,thm:QvalueControlsFDR}.) 
	\item[$\Pi_w$] the spike-and-slab prior \cref{eqn:def:SpikeAndSlabPrior}, under which $\theta_i=0$ with probability $1-w$ and is drawn from some (implicitly defined) density $\gamma$ with probability $w$, independently of the other $\theta_j$. 
	\item[$\Pi_w(\cdot \mid X)$] the induced posterior on $\theta$, see before \cref{eqn:def:Lvals}.
	\item[$\phi,g$] the standard Gaussian density and the quasi-Cauchy density \\ $g(x)=(2\pi)^{-1/2}x^{-2}(1-e^{-x^2/2})$ which respectively are the laws of $X_i$ under $\theta_i=0$ and under $\theta_i\sim \gamma$.
	\item[$\bar{\Phi},\bar{G}$] the upper tail distributions for $\phi,g$, e.g.\ $\bar{\Phi}(x)=\int_x^\infty \phi(t)\dt$.
	\item[$\ell_{i,w}(X)=$]$=\Pi_w(\theta_i=0\mid X) =\ell(w;X_i)=\frac{(1-w)\phi(X_i)}{(1-w)\phi(X_i)+wg(X_i)}.$ Also just denoted $\ell_{i,w}$ in places.
	\item[$q_{i,w}(X)$]$=q(w;X_i)=\frac{(1-w)\bar{\Phi}(\abs{X_i})}{(1-w)\bar{\Phi}(\abs{X_i})+w\bar{G}(\abs{X_i})}$ as in \cref{eqn:def:qvals}. Also just denoted $q_{i,w}$ at times.
	\item[$L(w)$] the log-likelihood \cref{eqn:def:logLikelihood}, $S(w)=L'(w)=\sum_{i=1}^n \beta(X_i)/(1+w\beta(X_i))$ the score function.
	\item[$\beta(x)$]$=(g/\phi)(x)-1$.
	\item[$\zeta(w)$]$=\beta^{-1}(1/w),~w\in(0,1]$.
	\item[$\xi$]$=(\phi/g)^{-1},$ $\chi=(\bar{\Phi}/\bar{G})^{-1}$. 
	\item[$\tilde{m}(w)$]$=-E_0 [\beta(X)/(1+w\beta(X))]=-\int_{-\infty}^{\infty} \beta(t)/(1+w\beta(t)) \phi(t)\dt$.
	\item[$\tilde{m}_1(\tau,w)$]$=E_\tau[\beta(X)/(1+w\beta(X))]=\int_{-\infty}^{\infty} \beta(t)/(1+w\beta(t))\phi(t-\tau)\dt$.
	\item[$r(w,t)$]$=wt(1-w)^{-1}(1-t)^{-1}$.
	\item[$F_w(x)$]$=P_{\theta_0=0}(\ell_{i,w}(X)<x)=2\bar{\Phi}(\xi(r(w,x)))$.
	\item[$\FDP,\FDR,\FNR$] the usual false discovery proportion, false discovery rate, and false negative rate, see \cref{eqn:def:FDP,eqn:def:FDR,eqn:def:FNR}. 
	\item[$\BFDR$] the Bayesian FDR, i.e.\ the FDR averaged over draws $\theta$ from the prior (see \cref{eqn:def:BFDR}).
	\item[$\postFDR_w(\vphi)$]$=E_{\theta \sim \Pi_w(\cdot \mid X}[\FDP(\vphi;\theta)]=\frac{\sum_{i=1}^n \ell_{i,w} \vphi_i}{1\vee(\sum_{i=1}^n \vphi_i)}.$
	\item[$\hat{w}$]$=\argmax_{w\in [1/n,1]} L(w)$, the maximum likelihood estimator for $w$.
	\item[$w_\pm$] quantities which will be used to upper and lower bound $\hat{w}$ with high probability, defined in \cref{eqn:def:w+,eqn:def:w-}.
	\item[$\hat{\lambda}$]$=\sup\braces{\lambda\in[0,1] : \postFDR_{\hat{w}} (\vphi_{\lambda,\hat{w}}) \leq t}.$
	\item[$\lambda_\pm$] quantities which will be used to upper and lower bound $\hat{\lambda}$ with high probability, defined in \cref{eqn:def:lambda+,eqn:def:lambda-}.
	\item[$\vphi_{\lambda,w}(X)$]$=(\II\braces{\ell_{i,w}(X)<\lambda})_{i\leq n}$.
	\item[$\vphi^{\Cl}$]$=\vphi_{\hat{\lambda},\hat{w}}$ (see also the beginning of \cref{sec:Bayesian-multiple-testing-procedures} for a direct definition).
	\item[$\vphi^{\qval}$]$=(\II\braces{q_{i,\hat{w}}<t})_{1\leq i\leq n}$.
	\item[$V_{\lambda,w}$]$=\#\braces{i\not \in S_0 : \ell_{i,w}<\lambda}$ the number of false discoveries made by $\vphi_{\lambda,w}$.
	\item[$V_w'$]$=\#\braces{i\not \in S_0 : q_{i,w}<t}$.
	\item[$\nu_n,\delta_n,\rho_n,\eps_n$] see \cref{eqn:def:nu,eqn:def:rho,eqn:def:eps,eqn:def:delta} ($\nu_n=(s_n/\log s_n)^{-1/2}$, $\delta_n=(\log (n/s_n))^{-1}$, $\eps_n=\delta_n\log\log(n/s_n)$, $\rho_n=e^{-v_n^2/9}$ with $v_n$ the signal strength as in \cref{eqn:def:StrongSignalClass}). In the setting of \cref{thm:FDRConvergenceWithRate}, $\eps_n$ is the largest of these asymptotically, see \cref{eqn:nu<delta,eqn:p<delta}.
	\item[$K_n$]=$\#\braces{i \in S_0 : \ell_{i,w_-}<\delta_n}$.
		\item[$\lesssim,\gtrsim,\asymp,\sim,\ll,o,O$:] For sequences $a_n,b_n$, $a_n\lesssim b_n$ or $a_n=O(b_n)$ means ($b_n\geq 0$ and) there exists a constant $C$ s.t. $\abs{a_n} \leq Cb_n$, and $C$ is independent of $n$ (and other arguments of $a,b$). $a_n\gtrsim b_n$ means $b_n\lesssim a_n$. $a_n\asymp b_n$ means $a_n\lesssim b_n$ and $a_n\gtrsim b_n$. $a_n\sim b_n$ means $a_n/b_n\to 1$, and $a_n\ll b_n$ or $a_n=o(b_n)$ means $a_n/b_n\to 0$. For functions $f,g$, all these relations are defined correspondingly.
\end{description}

\end{document}